\documentclass[11pt,reqno,a4paper]{amsart}

\usepackage{graphicx,wrapfig,latexsym,epsfig}
\usepackage{amsthm,amsfonts,amsmath,amssymb, amscd}
\usepackage{subcaption, mleftright}
\usepackage[margin=1in]{geometry}

\allowdisplaybreaks

\newtheorem{thm}{}[section]
\theoremstyle{theorem}
\newtheorem{theorem}[thm]{Theorem}
\newtheorem{corollary}[thm]{Corollary}
\newtheorem{lemma}[thm]{Lemma}

\newtheorem{definition}[thm]{Definition}

\newtheorem{remark}[thm]{Remark}

\newcommand{\veps}{\varepsilon}
\newcommand{\p}{\partial}

\graphicspath{{images/}}

\begin{document}

\title[Singularities of Equivariant Lagrangian Mean Curvature Flow]{Singularities of Equivariant Lagrangian Mean Curvature Flow}
\author[Albert Wood]{Albert Wood}

\begin{abstract}
	We study almost-calibrated, $O(n)$-equivariant Lagrangian mean curvature flow in $\mathbb{C}^n$, and prove structural theorems about the Type I and Type II blowups of finite-time singularities. In particular, we prove that any Type I blowup of such a singularity must be a special Lagrangian pair of transversely intersecting planes, any Type II blowup must be the Lawlor neck with the same asymptotes, and these blowups are independent of the choice of rescaling. We also give a partial classification of when singularities occur in the equivariant case, and examine the intermediate scales between the Type I and Type II models.
\end{abstract}

\maketitle

\section{Introduction}\label{sec-intro}

Since the resolution of the Poincar\'e conjecture by use of the Ricci flow with surgery by R. Hamilton and G. Perelman, the study of singularities of geometric flows has become a well-studied field. This is because in order to apply such surgery techniques, an understanding of the nature of possible singularities is necessary. For the mean curvature flow, much progress has been made; for example the work of T. Colding and W. Minicozzi \cite{colding:1} suggests that cylindrical singularities of hypersurface mean curvature flow are generic. Such singularities are Type I, meaning that a sequence of parabolic rescalings of the singularity converges subsequentially and locally smoothly to a self-similarly shrinking mean curvature flow. However, if we instead work with Lagrangian submanifolds (a preserved class under the flow), the singularities look very different. For example, in \cite{wang:1} M-T.\ Wang shows that Type I singularities such as the cylinder are impossible for a large open class of Lagrangian flows. Ambitious conjectures have been made about long-time convergence of Lagrangian mean curvature flow with surgeries, so an understanding of these singularities is important.

In this work we consider the question of which singularities are possible for the Lagrangian mean curvature flow, in the particular case of equivariant submanifolds. As we shall see, these submanifolds are necessarily Lagrangian (Lemma \ref{thm-s1lag}), and so they provide a model case for Lagrangian mean curvature flow. The equivariance essentially reduces the codimension of the flow to 1, since we may work with the quotient of the Lagrangian submanifold under the symmetry - the profile curve. We also restrict to `almost-calibrated' Lagrangian mean curvature flow (as defined in Section \ref{sec-prelim}), which is a condition on the Lagrangian angle. This is a necessary condition for the Thomas-Yau conjecture as reformulated by D.\ Joyce in \cite{joyce:1}, which approximately states that long-time existence and convergence of Lagrangian mean curvature flow is equivalent to a `stability condition', as is true for Hermitian-Yang-Mills flow. When combined with the equivariance, it can be seen that the almost-calibrated condition implies that the flow is non-compact and embedded (Lemmas \ref{lem-noncompact} and \ref{lem-embedded}). It is therefore most natural to work with flows $L_t$ with fixed, `planar' asymptotics, meaning that outside a compact ball $B_R$, the flow $L_t$ may be written as a graph over fixed planes through the origin.

We now go through our results in detail. The first main result, Theorem \ref{thm-typei}, provides a complete classification of singularities for our considered flow. Explicitly, we prove that any singularity of an almost-calibrated equivariant LMCF with planar asymptotics must occur at the origin, and that its Type I blowup is a special Lagrangian transverse pair of planes $P_1 \cup P_2$ which does not depend on the rescaling sequence (see Figure \ref{fig-type1type2}). We also provide a sufficient condition for when such a singularity occurs in Theorem \ref{thm-fts}.

\begin{theorem}\label{theorem1}
	Let $L_t$ be an almost-calibrated, connected $O(n)$-equivariant mean curvature flow in $\mathbb{C}^n$, with planar asymptotics.
	
	Then any finite-time singularity must occur at the origin. Additionally, a Type I blowup of such a singularity must be a special Lagrangian cone consisting of a transverse pair of planes $P_1 \cup P_2$ with identical Lagrangian angle, and the blowup does not depend on the rescaling sequence. The profile curves of these planes intersect at an angle of $\frac{\pi}{n}$.
	
	Additionally, if the asymptotes of the profile curve of $L_0$ span an angle $\beta > \frac{\pi}{n}$, then a finite time singularity must occur.
\end{theorem}
A similar result for equivariant Lagrangian spheres was shown in \cite{viana:1} by C. Viana - his examples are not almost-calibrated.

It is conjectured that any Type II blowup should have the same asymptotes as a Type I blowup, i.e. the `blowdown' of a Type II blowup should be a Type I blowup of the flow. Evidence for this is provided both by A. Savas-Halilaj and K. Smoczyk in \cite{Savas-Halilaj2018}, where it is shown that equivariant Lagrangian spheres develop Type II singularities with a double-density plane as the Type I blowup and the grim reaper as the Type II blowup, and by J.J.L. Vel\'azquez in \cite{velazquez:1}, in which he provides a MCF whose Type I blowup is the Simons' cone and whose Type II blowup is the unique minimal hypersurface tangent to it at infinity. Further analysis of the Vel\'azquez example was undertaken by N. Sesum and S-H. Guo in \cite{sesum:1}, including explicit estimates for the mean curvature and second fundamental form, and an examination of the intermediate scales.

Recently, B. Lambert, J. Lotay and F. Schulze proved in \cite{lls:1} that if the blowdown of a smooth Type II blowup is a pair of transverse planes $P_1\cup P_2$, the blowup must be the Lawlor neck, which is the minimal hypersurface with asymptotes $P_1 \cup P_2$ (unique up to scaling). Therefore if the above conjecture was true we would expect by Theorem \ref{theorem1} that every type II blowup of an almost-calibrated $O(n)$-equivariant flow to be the Lawlor neck. We verify this explicitly.
\begin{theorem}
	Let $L_t$ be an almost-calibrated, connected, $O(n)$-equivariant mean curvature flow in $\mathbb{C}^n$, with planar asymptotics.
	
	Then up to a translation, a Type II blowup of any finite-time singularity is the unique Lawlor neck with the same Lagrangian angle as the unique Type I blowup $P_1 \cup P_2$ and $\max|A|^2 = 1$. In particular, the asymptotes of this Type II blowup are the planes $P_1$ and $P_2$.
\end{theorem}
We also check the `intermediate scales', to confirm that there is no different behaviour in between the Type I and Type II scales - this is the content of Section \ref{sec-inter}. We prove that, using the same sequence of times as a Type II rescaling, if we use blowup factors smaller than the second fundamental form then we still obtain the blowup $P_1 \cup P_2$.
\begin{theorem}
	Let $L_t$ be an almost-calibrated, connected, $O(n)$-equivariant mean curvature flow in $\mathbb{C}^n$, with planar asymptotics. Assume that $L_t$ forms a singularity at the origin at time $t=0$, and let \[ L_\tau^{t_k,\lambda_k}\, := \, \lambda_k L_{t_k + T + \lambda_k^{-2}\tau} \] be a sequence of rescalings satisfying
	
	\[\quad \delta_k := \frac{\lambda_k}{A_k} \rightarrow 0, \quad -\lambda_k^2 t_k \rightarrow \infty,\]
	
	where $A_k := \max_{L_{t_k}}(|A|)$, and $0 > t_k \rightarrow 0$ satisfies (\ref{eq-typeiicondition}) for $p_k \equiv 0$. Then for any $R, \veps$ and finite time interval $I$, there exists a subsequence such that $L_\tau^{t_k,\lambda_k} \cap (B_R{\setminus} B_\veps)$ may be expressed as a graph over $P_1 \cup P_2$ for $\tau \in I$, and this graph converges in $C^{1;0}$ to 0.
\end{theorem}
The proofs of these theorems are contained in Section \ref{sec-s1mcf}. Section \ref{sec-prelim} is devoted to preliminary material about Lagrangians, mean curvature flow and blowups of singularities, and Section \ref{sec-s1} contains material on $O(n)$-equivariant submanifolds, including descriptions of the Lawlor neck and convergence theorems for sequences of equivariant submanifolds.

\subsection*{Acknowledgements} I am grateful to my PhD supervisor Felix Schulze for his support and guidance. Among the many others who have worked with or encouraged me, particular thanks are due to Chris Evans, Ben Lambert, Jason Lotay and Thomas Koerber, whose input has been invaluable. This work was supported by the Engineering and Physical Sciences Research Council [EP/L015234/1], The EPSRC Centre for Doctoral Training in Geometry and Number Theory (The London School of Geometry and Number Theory), University College London.

\begin{figure}[t]
	\centering
	\begin{subfigure}[b]{0.42\linewidth}
		\includegraphics[width=\linewidth]{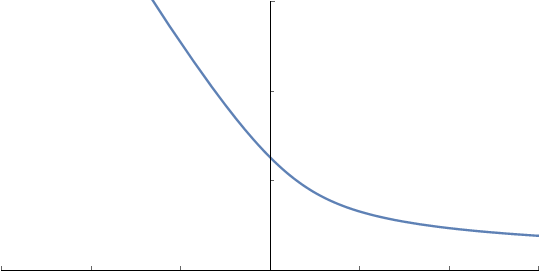}
		\caption{Initial Condition}
	\end{subfigure}
	\begin{subfigure}[b]{0.06\linewidth}
		$\,$
	\end{subfigure}
	\begin{subfigure}[b]{0.42\linewidth}
		\includegraphics[width=\linewidth]{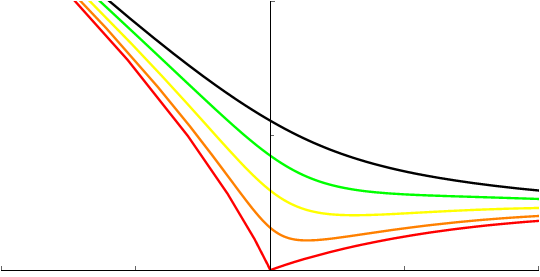}
		\caption{MCF}
	\end{subfigure}
	\caption{The profile curve of Neves' equivariant mean curvature flow in $\mathbb{C}^2$ spanning an angle $\beta = \tfrac{2\pi}{3}$, which forms a singularity at the origin.}
	\label{fig-nevesBetaLarge}
\end{figure}

\begin{figure}[t]
	\centering
	\begin{subfigure}[b]{0.42\linewidth}
		\includegraphics[width=\linewidth]{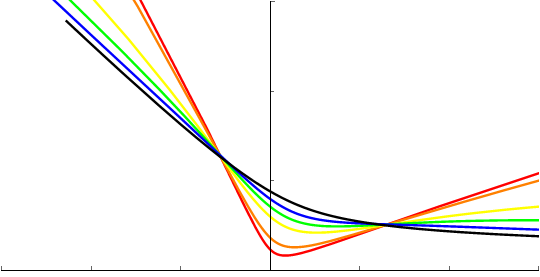}
		\caption{Type I Rescalings, at a fixed time}
	\end{subfigure}
	\begin{subfigure}[b]{0.06\linewidth}
		$\,$
	\end{subfigure}
	\begin{subfigure}[b]{0.42\linewidth}
		\includegraphics[width=\linewidth]{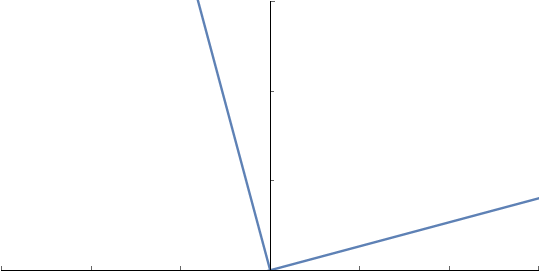}
		\caption{The Type I Blowup}
	\end{subfigure}
	\caption{Convergence of the profile curves of the Type I rescalings of Neves' equivariant mean curvature flow in $\mathbb{C}^2$ spanning an angle of $\beta = \tfrac{2\pi}{3}$.}
	\label{fig-nevesTypeIBlowup}
\end{figure}

\section{Preliminaries}\label{sec-prelim}

\subsection{Singularities of MCF}

A family of Riemannian immersions $F_t: M^m \hookrightarrow N^n$ where $t \in [t_0,T)$ is a \textbf{mean curvature flow} (often shortened to MCF) if 
\[ \frac{\partial F}{\partial t} \, = \, \vec{H}, \]
where $\vec{H}$ is the vector-valued mean curvature, the trace of the vector-valued second fundamental form:
\begin{align*}
	A_{ij} &:=  \overline\nabla^\perp_{\frac{\partial F}{\partial x^i}}\tfrac{\partial F}{\partial x^j},\\
	\vec{H} \, &:= \, H^{\alpha}\nu_{\alpha} \, = \, g^{ij}A_{ij}^{\alpha}\nu_{\alpha}.
\end{align*}
Our primary interest is in studying finite-time singularities of MCF. The following theorem, proven in the hypersurface case by Huisken (\cite{huisken:1}), helps us to understand what happens at these singular times.
\begin{theorem}
	Let $F_t:M^{m} \hookrightarrow N^n$ be a mean curvature flow of a closed submanifold, with corresponding second fundamental form $A_t$. If $T$ denotes the final time of existence, then
	\[ \limsup_{t \rightarrow T} \max_{p \in M_{t}}\left|A_t(p)\right|^2 \rightarrow \infty.\]
\end{theorem}
In the hypersurface case, this rate has a lower bound:
\[ \max_{p \in M}|A_t(p)|^2 \, \geq \, \frac{1}{2(T-t)}, \]
so the `best case' scenario is therefore one in which  $\max(|A|^2)$ has a rate of increase of $\frac{C}{2(T-t)}$ for some constant $C$. This inspires the following definition:
\begin{definition}
	A mean curvature flow $F_t:M \rightarrow N$ has a \textbf{Type I singularity} at time $T$ if, for a constant $C$,
	\[ \max_{p \in M}|A_t(p)|^2 \, \leq \, \frac{C}{T-t}.\]
	Otherwise, it is a \textbf{Type II singularity}.
\end{definition}
Take a flow $F_t:M \rightarrow \mathbb{R}^n$ with a Type I singularity at the spacetime point $(x,T)$, and consider the parabolic rescaling around this point,
\[ F^{\lambda}_s \, := \, \lambda \left ( F_{T + \lambda^{-2}s} - x \right ),  \]
which can be shown to be a mean curvature flow with time coordinate $s$. Taking a sequence $\lambda_i \rightarrow \infty$, we can use the bound on $|A|$ to show that the flows $F^{\lambda_i}_s$ converge subsequentially and locally smoothly to an ancient and self-similarly shrinking mean curvature flow $F^{\infty}_s$, which we call a \textbf{Type I blowup} of $F_t$. Note that the blowup may depend on the sequence chosen.

There are two very important tools that will help us understand the Type I blowup, namely the monotonicity formula of Huisken \cite{huisken:1} and the regularity theorem of White \cite{white:1}. In what follows, we will need the following modified backwards heat kernel:
\[ \Phi_{(x_0,t_0)}(x,t) \, := \, \left( 4\pi (t_0 - t)\right)^{-\frac{m}{2}} e^{-\frac{|x-x_0|^2}{4(t_0-t)}} .\]
\begin{theorem}[Huisken's Monotonicity Formula] \label{thm-monotonicity}
	Let $F_t: M^m \hookrightarrow \mathbb{R}^n$ be a smooth solution of MCF, where $F_t(M)$ has bounded area ratios. Then:
	\[ \frac{d}{dt} \int_{F_t} \Phi_{(x_0,t_0)} \,\, d\mathcal{H}^m \, = \, -\int_{F_t} \left| \vec{H} + \frac{(x-x_0)^{\perp}}{2(t_0-t)}\right|^2 \Phi_{(x_0,t_0)} \,\, d\mathcal{H}^m.   \]
	More generally, if $f_t$ is any smooth function with polynomial growth at infinity, then
	\[ \frac{d}{dt} \int_{F_t} f_t \, \Phi_{(x_0,t_0)} \,\, d\mathcal{H}^m \, = \, \int_{F_t} \left( \frac{df_t}{dt} - \Delta f_t - f_t\left| \vec{H} + \frac{(x-x_0)^{\perp}}{2(t_0-t)}\right|^2 \right) \Phi_{(x_0,t_0)} \,\, d\mathcal{H}^m.   \]
\end{theorem}
The integral quantity in the monotonicity formula satisfies a very useful scale invariance:
\begin{lemma}\label{lem-huiskenscaleinvariance}
	If $F_s^\lambda(N) = N^{\lambda}_s$ is a parabolic rescaling of the flow $F_t(N) = N_t$ around the point $(x_0,t_0)$ with $s = \lambda^2(t - t_0)$, then for $X = (\overline{x}, \overline{t})$, and any function $f$ on $N$,
	\[ \int_{N_t} f \, \Phi_X \, d\mathcal{H}^n \, = \, \int_{N_s^\lambda} f \, \Phi_{\left( \lambda(\overline{x}- x_0), \lambda^2(\overline{t}-t_0)\right)} \, d\mathcal{H}^n  .\] 
\end{lemma}
\begin{proof}
	\begin{align*}
		\int_{N_t} f \,\Phi_{X} \, d\mathcal{H}^n \, &= \, \int_{N_t} f \,(4\pi (\overline{t}-t))^{-\tfrac{n}{2}} \cdot e^{-\frac{|\overline{x}-x|^2}{2(\overline{t}-t)} } \, d\mathcal{H}^n(x)\\
		&= \, \int_{\lambda(N_t - x_0)} f \,\lambda^{-n}(4\pi (\overline{t}-t))^{-\tfrac{n}{2}} \cdot e^{-\frac{|\overline{x}-(\lambda^{-1}x+x_0)|^2}{2(\overline{t}-t)} } \, d\mathcal{H}^n(x)\\
		&= \, \int_{\lambda(N_{\lambda^{-2}s + t_0} - x_0)} f \,(4\pi \left( \lambda^2(\overline{t}-t_0) - s \right))^{-\tfrac{n}{2}} \cdot e^{-\frac{|\lambda(\overline{x}-x_0) - x|^2}{2(\lambda^2(\overline{t}-t_0) - s)} } \, d\mathcal{H}^n(x)\\
		&= \, \int_{N_s^\lambda} f \,\Phi_{\left( \lambda(\overline{x}- x_0), \lambda^2(\overline{t}-t_0)\right)} \, d\mathcal{H}^n .
	\end{align*}
\end{proof}
Huisken's monotonicity formula inspires the following quantity, known as the \textbf{Gaussian density ratio}, for a spacetime point $X = (x_0,t_0)$:
\[ \Theta(F, X, r)  \, := \, \int_{F_{t_0-r^2}} \Phi_{X}(x,t_0-r^2) \,\, d\mathcal{H}^m. \]
The monotonicity formula implies that this quantity is increasing in $r$. The Gaussian density ratio is useful for controlling the curvature of our flow - this is the content of White's regularity theorem. Denote by $P_r(x,t)$ the \textbf{parabolic cylinder} $B_r(x) \times (t - r^2, t]$.
\begin{theorem}[White's Regularity Theorem] \label{thm-white}
	There exist $\veps >0$, $C >0$ depending on $n$ such that if $F_t: M^{m} \rightarrow \mathbb{R}^n$ is a smooth mean curvature flow, and if
	\[\forall X \in P_r(X_0), \quad \Theta(M,X,r) \, \leq 1 + \veps, \]
	then
	\[ \sup_{X \in P(X_0,\frac{r}{2})} |A(X)| \, \leq \frac{C}{r}. \]
\end{theorem}
It follows from standard theory of elliptic PDEs that under the same conditions, for similar universal constants $C_k$,
\[ \sup_{X \in P(X_0,\frac{r}{4})} | \nabla^k A| \, \leq \, \frac{C_k}{r^{m+1}}. \]
This regularity theorem is useful for proving smooth convergence to the Type I model. If we know that the density of the limiting model is close to 1, then we can use convergence in measure to show that the density of our flows is also close to 1, and White will then imply that our curvatures are bounded, so the convergence is smooth.

The Type I blowup procedure results only in a weak flow for Type II singularities. The trick to resolving these singularities smoothly is to take a sequence of spacetime points $(x_i,t_i)$ maximising the second fundamental form $A_{t_i}$, and then to perform a parabolic rescaling with factor $|A_{t_i}(p_i)|$ around that point to normalise its value to $1$. There is a complication however - in order to have a smooth convergence to the blowup, we need control on $|A|$ for a period of time around $t_i$. To achieve this, we choose a sequence of times $t_k \in [0, T-\tfrac{1}{k}]$ and points $p_k \in M$ such that:
\begin{align}
|A_{t_k}(p_k)|^2\left( T - \tfrac{1}{k} - t_k \right) \, = \, \max_{t \in [0,T-\tfrac{1}{k}], \, p \in M} \left( |A_t(p)|^2 \left( T- \tfrac{1}{k} - t \right) \right). \label{eq-typeiicondition}
\end{align}
Note that the second fundamental form at time ${t_k}$ is maximised at the point $p_k$. It then follows from the Type II condition (see e.g. \cite{mantegazza:1}) that one can choose a subsequence such that:
\begin{itemize}
	\item $|A_{t_k}(p_k)| \rightarrow \infty$ monotonically,
	\item $|A_{t_k}(p_k)|^2\left( T- \tfrac{1}{k} - t_k\right) \rightarrow \infty$,
	\item $p_k \rightarrow p$ for some $p \in M$,
\end{itemize}
where the last point is immediate if our manifold is compact, and otherwise must be proven.

Now we rescale the flow $F_t$, restricted to the time interval $[0,T-\tfrac{1}{k}]$, parabolically with factor $A_k := |A_{t_k}(p_k)|$ around $(x_k, t_k)$:
\[ F^{(x_k,t_k)}_\tau(p) \, := \, A_k \left( F_{t_k + A_k^{-2}\tau}(p) - x_k \right) \]
where $x_k := F_{t_k}(p_k)$. This flow is defined for $\tau \in I_k := \left[ -A_k^2 t_k, A_k^2\left( T-\tfrac{1}{k} - t_k \right) \right]$. These rescalings converge locally smoothly to a limiting eternal flow (see e.g. \cite{mantegazza:1})- a \textbf{Type II blowup}, and the value of $|A|$ for this blowup takes a maximum of $1$ over time and space. By the definition of the rescalings, this maximum value is achieved at the spacetime point $(0,0)$.

\subsection{Lagrangian Submanifolds}

If we denote by $J$ and $\omega$ the standard complex and symplectic structures on $\mathbb{C}^n$,
\[ \omega \, = \, dx^1\wedge dy^1 + dx^2\wedge dy^2 + \cdots + dx^n\wedge dy^n,  \quad \omega(X,Y) = \left\langle JX,Y \right\rangle, \]
and by $\Omega$ the holomorphic volume form, given by \[\Omega \, = \, dz^1 \wedge dz^2 \wedge \cdots \wedge dz^n,\] then a smooth orientable $n$-dimensional submanifold $L$ is said to be \textbf{Lagrangian} if $\omega|_L \, \equiv \, 0$. It follows from a calculation (e.g \cite{harvey:1}) that $\Omega|_L \, = \, e^{i\theta}vol_L$, for some multivalued function $\theta$ known as the \textbf{Lagrangian angle}. If \\ $\{X_1,X_2,\ldots, X_n\} \subset \mathbb{C}^n$ are linearly independent vectors tangent to $L$ at a point $p \in L$, then the Lagrangian angle can be calculated (up to a multiple of $\pi$) as:
\begin{equation}\label{eqn-laganglecalc}
\theta(p) \, = \, \arg(\mbox{det}_{\mathbb{C}}(X_i^j)).
\end{equation}
If we ensure that $vol_L(X_1,X_2,\ldots,X_n)=1$, i.e. an orientation is chosen, then the Lagrangian angle is determined modulo $2\pi$ by this method.

If the Lagrangian angle is a single-valued function, then $L$ is known as a \textbf{zero-Maslov} Lagrangian - this is equivalent to the \textbf{Maslov class} $[d\theta] \in H^1(L)$ vanishing. If the Lagrangian angle takes only a single constant value, then we say $L$ is a \textbf{special Lagrangian}. These are particularly interesting because they are calibrated and therefore minimal.

\subsection{Lagrangian Mean Curvature Flow}
The foundational result of Lagrangian mean curvature flow is that mean curvature flow preserves the class of Lagrangian submanifolds (see \cite{smoczyk:4} for details).
\begin{theorem}\label{thm-lagmcf}
	Let $L_t$ be a mean curvature flow in $\mathbb{C}^n$ for $t \in [0,T)$, such that $L_{0}$ is a Lagrangian submanifold. Then:
	\begin{itemize}
		\item $L_{t}$ is a Lagrangian submanifold for all $t \in [0,T),$
		\item $\vec{H} \, = \, J\nabla \theta,$
		\item If $L_{0}$ is zero-Maslov, then $L_{t}$ is zero-Maslov also, and $\frac{d}{dt} \theta \, = \, \Delta \theta.$
	\end{itemize}
\end{theorem}
Singularities of LMCF have been studied extensively, for example by Neves in \cite{neves:1} and \cite{neves:2}. The first of these papers contains two important theorems for the zero-Maslov case. Theorem A tells us that any Type I blowup of a zero-Maslov LMCF looks like a union of special Lagrangian cones:
\begin{theorem}[Neves' Theorem A]\label{thm-a}
	If $L_0$ is a zero-Maslov class Lagrangian with bounded Lagrangian angle, then for any sequence of Type I rescaled flows $(L^i_s)_{s<0}$ at a singularity, with Lagrangian angle $\theta^i_s$, there exist a finite set $\{ \overline\theta_1,\ldots, \overline\theta_N \}$ and integral special Lagrangian cones $\{ L_1, \ldots, L_N \}$ such that on passing to a subsequence, for every $f \in C^2(\mathbb{R})$, $\phi \in C^{\infty}_c(\mathbb{C}^n)$ and $s<0$,
	\[ \lim_{i \rightarrow \infty} \int_{L^i_s}f(\theta^i_{s})\phi d\mathcal{H}^n \, = \, \sum_{j=1}^N m_j f(\overline\theta_j) \mu_j(\phi),\]
	where $\mu_j$, $m_j$ denote the Radon measure of the support and multiplicity of $L_j$ repectively. Furthermore, the set  $\{ \overline\theta_1,\ldots, \overline\theta_N \}$ doesn't depend on the sequence chosen.
\end{theorem}
Theorem B tells us that these cones in fact have the same Lagrangian angle, if we assume a couple of extra conditions. A Lagrangian submanifold is \textbf{almost-calibrated} if the Lagrangian angle has a range of less than $\pi$, explicitly:
\[ \exists \, \overline{\theta}, \, \, \exists \veps > 0 \quad s.t \quad \theta \in \left( \overline{\theta}-\frac{\pi}{2} + \veps, \, \overline{\theta} + \frac{\pi}{2} - \veps \right).\]
It is a strengthening of zero-Maslov. Almost-calibrated Lagrangian mean curvature flow may only form Type II singularities, as shown by M-T. Wang \cite{wang:1}:

\begin{theorem}\label{thm-ACLMCFsing}
	An almost-calibrated Lagrangian mean curvature flow $L_t$ in a K\"ahler-Einstein manifold $M$ cannot form a Type I singularity.
\end{theorem}

A Lagrangian submanifold is \textbf{rational} if some $a \in \mathbb{R}$,
\[ \lambda(H_1(L_0,\mathbb{Z})) \, \in \, \{ a2k\pi | k \in \mathbb{Z} \}, \]
for $\lambda := \sum_{i=0}^n x^i dy^i -y^i dx^i$ the Liouville form. The rational condition is a generalisation of exactness - the form $\lambda|_{L}$ being exact is precisely $L$ being rational with $a=0$.

Both of these conditions are preserved by mean curvature flow - for example preservation of `almost-calibrated' follows from the evolution equation for $\theta$ in Theorem \ref{thm-lagmcf}. See \cite{neves:1} for a proof of preservation of rationality.
\begin{theorem}[Neves' Theorem B]\label{thm-b}
	If $L_0$ is almost-calibrated and rational, then after passing to a subsequence of the rescaled flows $L^i_s$, with Lagrangian angle $\theta^i_s$, the following holds for all $R > 0$ and almost all $s < 0$.	
	
	For any convergent subsequence (in the Radon measure sense) $\Sigma^i$ of connected components of $B_{4R}(0)\cap L^i_s$ intersecting $B_R(0)$, there exists a special Lagrangian cone $L$ in $B_{2R}(0)$ with Lagrangian angle $\overline{\theta}$ such that for every $f \in C(\mathbb{R})$ and every $\phi \in C^{\infty}_c (B_{2R}(0))$,
	\[ \lim_{i \rightarrow \infty} \int_{\Sigma^i}f(\theta^i_{s})\phi d\mathcal{H}^n \, = \, m f(\overline{\theta}) \mu(\phi),\]	
	where $\mu$ and $m$ denote the Radon measure of the support of $L$ and the multiplicity respectively.
\end{theorem}
An important aspect of Theorem \ref{thm-b} to note is that it concerns a sequence of \emph{connected components} of $L^i_s \cap B_{4R}$ in the Type I rescaling, which corresponds to a sequence of connected components in a shrinking ball for the original flow. To ensure that we only get one special Lagrangian in the limit, we must ensure we are only looking at a single connected component in this ball.

\section{$O(n)$-Equivariant Submanifolds in $\mathbb{C}^n$} \label{sec-s1}

An \textbf{$O(n)$-equivariant submanifold} is a submanifold $L\subset \mathbb{C}^n$ that may be expressed as the image of a function
\begin{equation} 
	L: M^1\times S^{n-1} \rightarrow \mathbb{C}^n = \mathbb{R}^n \times \mathbb{R}^n, \quad L(s, \alpha) \, = \, \left( a(s) \alpha, \, b(s) \alpha \right),
\end{equation}
for some smooth functions $a, b: M^1 \rightarrow \mathbb{R}$, where $M$ is a $1$-dimensional manifold. $L$ is invariant under the $O(n)$ action
\begin{align*}
	O(n) \circlearrowright \mathbb{C}^n, \quad \quad 
	A \big ( ({x},{y}) \big ) = 	
	(A{x},A{y})
\end{align*}
for $x, y \in \mathbb{R}^n$, $A \in O(n)$. Of particular importance is that $L(s, \alpha) = -L(s,-\alpha)$, implying that $L$ has reflective symmetry through the origin. The \textbf{profile curve}
\begin{equation}
	l: M^1 \rightarrow \mathbb{C}, \quad l(s) = a(s) + ib(s)
\end{equation}
can therefore be chosen to have reflective symmetry across the origin. We will make this choice (we can think of the profile curve as the intersection $L_t \cap (\mathbb{C}{\times}\{0\}^{n-1})$, if we identify $\mathbb{C}$ with $\mathbb{C}{\times}\{0\}^{n-1}$). Since we demand that the manifold $L$ is connected, if $l$ passes through the origin then we must have a single connected component, and if it does not, then $l$ has two connected components $\gamma$ and $-\gamma$. In the case that $l$ passes through the origin, note that by the reflective symmetry, $\vec{H}=0$ there.

We will now prove some results regarding equivariant submanifolds, in particular that they are Lagrangian, and that the almost-calibated condition implies non-compactness and embeddedness.

\begin{lemma}\label{thm-s1lag}
	An immersed $O(n)$-equivariant surface $L \subset \mathbb{C}^n$ is a Lagrangian submanifold.
\end{lemma}
\begin{proof}
	We must show that $\omega|_{L} \equiv 0$. If we pick a local coordinate system $(\sigma^1,\ldots, \sigma^{n-1})$ for $S^{n-1}$, the derivatives of $L$ are given by 
	\begin{align*}
		\frac{\p L}{\p s} \, &= \,\left( a'(s) \alpha, \, b'(s) \alpha \right), \quad
		\frac{\p L}{\p \sigma^i} \, = \,\left( a(s) \frac{\p\alpha}{\p \sigma^i}, \, b(s) \frac{\p\alpha}{\p \sigma^i} \right),
	\end{align*}
	where we identify $\mathbb{C}^n$ with $\mathbb{R}^n\times\mathbb{R}^n$.	Remembering that the almost complex structure $J$ in $\mathbb{C}^n$ is given by $J(x,y) \, = \, (-y,x)$, we can then calculate $\omega|_{L}$:
	\begin{align*}
		\omega\left(\frac{\p L}{\p \sigma^i},\frac{\p L}{\p \sigma^j}\right) \, &= \, \left\langle J \left( \frac{\p L}{\p \sigma^i}\right), \, \frac{\p L}{\p \sigma^j} \right\rangle \\
		&= \,\left( -b(s) \frac{\p\alpha}{\p \sigma^i}, \, a(s) \frac{\p\alpha}{\p \sigma^i} \right)\cdot\left( a(s) \frac{\p\alpha}{\p \sigma^j}, \, b(s) \frac{\p\alpha}{\p \sigma^j} \right) = 0, \\
		\omega\left(\frac{\p L}{\p s},\frac{\p L}{\p \sigma^j}\right) \, &= \, \left( -b'(s) \alpha, \, a'(s) \alpha \right) \cdot \left( a(s) \frac{\p\alpha}{\p \sigma^j}, \, b(s) \frac{\p\alpha}{\p \sigma^j} \right)
		= \, 0, \quad\\
		\omega\left(\frac{\p L}{\p s},\frac{\p L}{\p s}\right) \, &= \left( -b'(s) \alpha, \, a'(s) \alpha \right) \cdot \left( a'(s) \alpha, \, b'(s) \alpha \right)
		= \, 0,
	\end{align*}
	where for the second line we use $\alpha \cdot \frac{\p\alpha}{\p \sigma^j} = 0$, which is a property of the sphere $S^{n-1}$.
\end{proof}
Since equivariant submanifolds are Lagrangian, we may consider their Lagrangian angle $\theta$. Locally and up to a multiple of $2\pi$, the Lagrangian angle is given by (\ref{eqn-laganglecalc}) to be
\begin{align}\label{eq-langle}
	\theta(s,\alpha) \, = \, (n-1)\arg(l(s)) \, + \, \arg(l'(s)).
\end{align}
Note that this implies the Lagrangian angle is well-defined for the profile curve $l$. 

\begin{lemma}\label{lem-noncompact}
	A connected, $O(n)$-equivariant, embedded, zero-Maslov Lagrangian submanifold $L$ of $\mathbb{C}^n$ is non-compact and rational. Moreover, if the profile curve contains the origin, then $L \cong \mathbb{R}^n$, and if not, $L \cong \mathbb{R}\times S^{n-1}$.
\end{lemma}
\begin{proof}
	Assume that the profile curve $l$ is compact, for a contradiction. Consider for simplicity a connected component $\gamma$ of the profile curve $l$, which is embedded by assumption and homeomorphic to a circle. Firstly we claim that $O \notin \gamma$.
	Otherwise, we could parametrise $\gamma$ by unit-speed such that 
	\[\gamma(0) = O, \quad \gamma(s) = -\gamma(-s) \]
	by the $O(n)$-equivariance. But then since $\gamma$ is compact, there must exist $S > 0$ such that $\gamma(S) = \gamma(-S)$, which implies that $\gamma(S) = O$. This contradicts embeddedness of $\gamma$.
	
	Now consider the following integral:
	\begin{align*}
		\int_\gamma d\theta \, &= \, \int \frac{\p}{\p s} \arg(\gamma'(s))\,ds \, + \, (n-1)\int \frac{\p}{\p s}\arg(\gamma(s)) \, ds\\
		&= \, 2\pi T(\gamma) \, + \, 2(n-1)\pi W_O(\gamma),
	\end{align*}
	where $T(\gamma)$ is the turning number, and $W_O(\gamma)$ is the winding number around the origin. Since $\gamma$ is embedded, it follows from standard theory that $T(\gamma) \in \{-1,1\}$, and $W(\gamma) \in \{T(\gamma),0\}$ (depending on whether the origin is contained in $\gamma$ or not). It follows that $[d\theta][\gamma] \, \neq \, 0$, contradicting the zero-Maslov assumption.
	
	Finally, since the submanifold has been shown to be non-compact, the domain for the profile curve $M^1$ must be homeomorphic to $\mathbb{R}$. So, by the equivariance, $L \cong \mathbb{R}^n$ or $\mathbb{R}\times S^{n-1}$ (depending on whether the profile curve contains the origin or not respectively). Since both of these have first homology generated by at most one element, rationality of $L$ follows from the definition.
\end{proof}
\begin{figure}[h]
	\centering
	\includegraphics[scale=0.65]{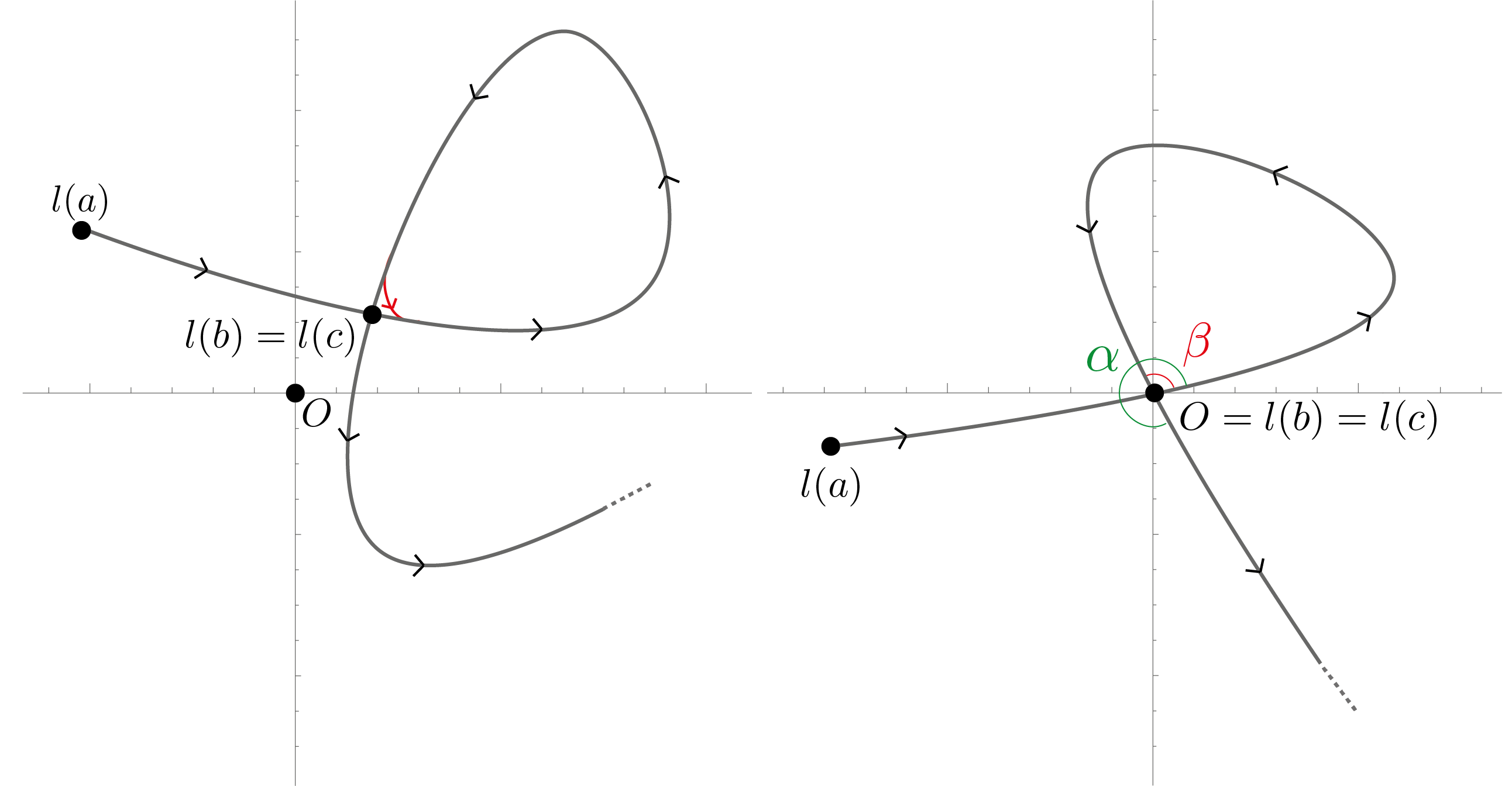}
	\caption[Diagram for the proof of Lemma \ref{lem-embedded}]{The proof of Lemma \ref{lem-embedded}. The case where the intersection is away from the origin is  depicted on the left, with the smoothing highlighted in red, and the case where the intersection is at the origin is depicted on the right.}
	\label{fig-embeddeddiagram}
\end{figure}
\begin{lemma}\label{lem-embedded}
	A connected, $O(n)$-equivariant, almost-calibrated Lagrangian submanifold $L$ of $\mathbb{C}^n$ must be embedded, non-compact and rational.
\end{lemma}
\begin{proof}
	Consider a connected component of the profile curve, $l:M^1 \rightarrow \mathbb{C}$. Parametrising $M^1$ by the real numbers (or a quotient of $\mathbb{R}$ if $M^1$ is compact), choose $a \in M^1$ sufficiently small such that there exists $c > a$ with $l([a,c])$ not embedded. Taking $c$ as the infimum of its possible values, it follows that there exists $a\leq b\leq c$ such that $l(b)=l(c)$, and $l([b,c])$ is a closed loop.
	
	We consider two possibilities: either $l(b)=l(c)=O$, or $l(b),l(c)\neq O$ (see Figure \ref{fig-embeddeddiagram}). In the latter case, at the point $l(b)$ we may smooth the curve off to create a closed, embedded loop $\gamma$. By (\ref{eq-langle}), we can do this so that the Lagrangian angle of the smoothed section lies in the interval $(\theta(b)-\varepsilon,\theta(c)+\varepsilon)$ for any given $\varepsilon$. This produces a compact, embedded, almost-calibrated loop, which contradicts Lemma \ref{lem-noncompact}. In the other case, assume without loss of generality that the loop is oriented anticlockwise, and define
	\[\alpha \, := \, \int_{l([b,c])} \frac{\partial}{\partial s} \arg(\gamma'(s))\,ds,\quad\quad\quad \beta \, := \, \int_{l([b,c])} \frac{\p}{\p s}\arg(\gamma(s)) \, ds.\]
	Note that due to the orientation and embeddedness, $\alpha, \beta$ must be positive, and $\alpha - \beta = \pi$. It therefore follows that
	\begin{align*}
		\int_{[b,c]} d\theta \, &= \, \int_{l([b,c])} \frac{\partial}{\partial s} \arg(\gamma'(s))\,ds \, + \, (n-1)\int_{l([b,c])} \frac{\partial}{\partial s}\arg(\gamma(s)) \, ds\\
		&= \, \alpha + (n-1)\beta\\
		&= \, \pi + n\beta\\
		&> \, \pi,
	\end{align*}
	which is a contradiction to the almost-calibrated condition.
\end{proof}
We remark that there do exist non-embedded zero-Maslov equivariant curves, for example any equivariant Lagrangian with a `figure 8' profile curve through the origin, such as the Whitney sphere studied in \cite{Savas-Halilaj2018}.

\subsection{$O(n)$-Equivariant Special Lagrangians}

We would now like to characterise the $O(n)$-equivariant special Lagrangian cones in $\mathbb{C}^n$ (potential Type I blowup models) and the $O(n)$-equivariant smooth special Lagrangians (potential Type II blowup models).
\begin{lemma}\label{lem-speciallagcones}
	The only $O(n)$-equivariant special Lagrangian cones in $\mathbb{C}^n$ are unions of special Lagrangian planes, with profile curve consisting of unions of lines through the origin.
\end{lemma}
\begin{proof}
	The only special Lagrangian cones in $\mathbb{C}$ are lines through the origin, so this follows from the equivariance and (\ref{eq-langle}).
\end{proof}

\begin{lemma}\label{lem-speciallag}
	The only $O(n)$-equivariant surfaces in $\mathbb{C}^n$ with constant Lagrangian angle of $\overline{\theta}$ are those with profile curves given by either lines through the origin, or the parametrisation
	\begin{align*}
		r(\alpha) \, &= \, \frac{B}{\sqrt[n]{\sin(\overline{\theta}-n\alpha)}}
	\end{align*}
	for $B \geq 0$.
	In the latter case, these are known as \textbf{Lawlor necks}.
\end{lemma}
\begin{proof}
	If a connected component $\gamma$ of the profile curve is not a line through the origin, then there is an open interval on which it may be parametrised by angle, i.e. as $\gamma(\alpha) \, = \, r(\alpha)e^{i\alpha}.$ Then on this interval,
	\begin{align*}
		\dot{\gamma}\, = \, (\dot{r} + ir)e^{i\alpha}\quad	\implies \quad \overline{\theta} \, = \, n\alpha + \cot^{-1}(\tfrac{\dot{r}}{r}).
	\end{align*}
	by equation (\ref{eq-langle}). Integrating this gives the expression in the statement, and since this expression is valid until the value of $r$ diverges to $\infty$, the entire connected component may be parametrised in this way.
\end{proof}

\subsection{Limits of $O(n)$-Equivariant Submanifolds}\label{sub-s1lim}

When considering Type I and Type II blowups, we will be trying to understand the limit of sequences of submanifolds, $L^i$. Since they are translations and dilations of equivariant submanifolds, they will have rotational symmetry, but the centres of rotation $x_i$ may not be the origin (though without loss of generality, we will be able to assume that $x_i \in \mathbb{C}{\times}\{0\}^{n-1}$). There are two possible behaviours: either $|x_i|$ stays bounded, or diverges to infinity. These two cases will correspond to equivariance and translation invariance respectively for the limiting object.

\newcommand{\measurerestr}{%
	\,\raisebox{-.127ex}{\reflectbox{\rotatebox[origin=br]{-90}{$\lnot$}}}\,%
}

We formalise and prove these statements. Consider for the rest of this section a sequence $L^i$ of submanifolds of $\mathbb{C}^n$, which converge in the sense of Radon measures to $\mu_\infty \, := \, m\,\mathcal{H}^n\measurerestr L^{\infty}$, where $m$ is a multiplicity function and $L^\infty$ is the supporting set. Explicitly, for all $\phi \in C^{\infty}_c(\mathbb{C}^n)$, denoting the underlying Radon measures of the $L^i$ by $\mu_i$,
\[ \mu_i(\phi) = \int_{L^i} \phi \, d\mathcal{H}^n \, \rightarrow \, \mu_\infty(\phi) = \int_{L^\infty} m \phi \, d\mathcal{H}^n. \]
Assume that $L^i$ is a translation of an $O(n)$-invariant submanifold by $x_i \in \mathbb{C}{\times}\{0\}^{n-1}$, therefore invariant under the rotation mappings
\begin{align}
	R_{x_i} \, &: \, S^{n-2}\times\mathbb{R}\times\mathbb{C}^n \rightarrow \mathbb{C}^n,\notag\\
	R_{x_i} (\alpha, \lambda, y) \, &:= \, 
	\mleft(
	\begin{array}{c|c}
		\begin{matrix}
			\cos \left( \frac{\lambda}{|y-x_i|}\right) && -\sin \left( \frac{\lambda}{|y-x_i|}\right)\\
			\sin \left( \frac{\lambda}{|y-x_i|}\right) && \cos \left( \frac{\lambda}{|y-x_i|}\right)
		\end{matrix} & 0 \\
		\hline
		0 & Id
	\end{array}
	\mright)(y - x_i) + x_i, \label{eq-defRrotation}
\end{align}
where $\alpha \in S^{n-1}\cap \left(\{0\}{\times}\mathbb{R}^{n-1}\right) \cong S^{n-2}$ is an equatorial element of $S^{n-1}$ (the direction of rotation), $\lambda \in \mathbb{R}$ is a distance factor, and for the matrix we have used an orthogonal basis of $\mathbb{R}^n$ starting with $e_1$ and $\alpha$. Note that keeping $\alpha$ constant and varying $\lambda$ creates a 1-parameter family of rotations that corresponds to the rotations of the $S^1 \subset S^{n-1}$ containing $\alpha$ and $e_1$. Define also the translation map

\begin{equation}
	T_{x_i} (\alpha, \lambda, y) \, := y - \lambda 
	\mleft(
	\begin{array}{c|c}
		\begin{matrix}
			0 && -1\\
			1 && 0  
		\end{matrix} & 0 \\
		\hline
		0 & 0
	\end{array}
	\mright)
	\frac{x_i}{|x_i|}, \label{eq-defTtranslation}
\end{equation}
using the same basis.
\begin{lemma}\label{lem-s1convergence}
	If $|x_i| \rightarrow \infty$, then we may pass to a subsequence such that as $i \rightarrow \infty$, 
	\[ \tfrac{x_i}{|x_i|} \rightarrow v, \quad \quad R_{x_i} \rightarrow T_{v} \]
	in $C^{\infty}_{loc}$. 
	If $|x_i|$ remains bounded, then we may pass to a subsequence such that as $i \rightarrow \infty$,
	\[ x_i \rightarrow x, \quad \quad R_{x_i} \rightarrow R_{x} \]
	in $C^{\infty}$.
\end{lemma}
\begin{proof}
	This is clear in the $|x_i|$ bounded case. If $|x_i| \rightarrow \infty$, then first pass to a subsequence such that $\tfrac{x_i}{|x_i|} \rightarrow v$. Fixing a compact region of the domain $U\subset S^{n-2}{\times}\mathbb{R}{\times}\mathbb{C}^n$ and taking $(\alpha, \lambda, y) \in U$,
	\begin{align*}
		\frac{\lambda}{|y-x_i|} \, &\rightarrow \, 0\\
		\implies R_{x_i}(\alpha, \lambda, y) &- T_{x_i}(\alpha, \lambda, y) \, = 
		\mleft( 
		\begin{array}{c|c}
			\begin{matrix}
				\cos \big( \frac{\lambda}{|y-x_i|}\big) - 1 && -\sin \big( \frac{\lambda}{|y-x_i|}\big)\\
				\sin \big( \frac{\lambda}{|y-x_i|}\big) && \cos \big( \frac{\lambda}{|y-x_i|}\big) - 1
			\end{matrix} & 0 \\ \hline
			0 & 0
		\end{array}
		\mright) y \, + \, \\ &\mleft( 
		\begin{array}{c|c}
			\begin{matrix}
				1 - \cos \big( \frac{\lambda}{|y-x_i|}\big) && \sin \big( \frac{\lambda}{|y-x_i|}\big) - \frac{\lambda}{|x_i|}\\
				-\sin \big( \frac{\lambda}{|y-x_i|}\big) + \frac{\lambda}{|x_i|} && 1 - \cos \big( \frac{\lambda}{|y-x_i|}\big)
			\end{matrix} & 0 \\ \hline
			0 & 0
		\end{array}
		\mright)x_i \longrightarrow 0,
	\end{align*}
	where all convergences are uniform in $U$. Similarly, the derivatives converge uniformly. Since also
	\[|T_{x_i}(\alpha, \lambda, y) - T_{v}(\alpha, \lambda, y)| \, \rightarrow 0 \] in $C^{\infty}$, the result follows.
\end{proof}
What kind of invariance can we deduce for $\mu_\infty$ from this lemma? It is immediate that in both cases we can extract a measure-theoretic invariance. For example in the $|x_i|$ unbounded case, taking $\phi \in C^{\infty}_c$:
\begin{align*}
	\mu_\infty (\phi \circ T_v(\alpha, \lambda, \cdot)) \, &= \, \lim_{i \rightarrow \infty} \mu_i(\phi \circ R_{x_i}(\alpha, \lambda, \cdot)) \ \, =  \, \lim_{i \rightarrow \infty} \mu_i(\phi ) \, = \, \mu_{\infty} (\phi ).
\end{align*}

It follows that if $L^\infty$ is a cone smooth away from the origin, or indeed a smooth manifold, then we have invariance of the supporting set, as well as of the multiplicity function. 

One of the most useful aspects of $O(n)$-equivariant smooth manifolds is that they are characterised by the intersection with $\mathbb{C} {\times} \{0\}^{n-1}$. In particular, it is convenient to replace the $\mathcal{H}^n$ (Hausdorff) measure of our submanifolds with the $\mathcal{H}^1$ measure of their intersection with $\mathbb{C} {\times} \{0\}^{n-1}$. We wish to do this also with our limit $\mu_\infty$, in the case where $L^\infty$ is a cone. In place of the Hausdorff measures $\mathcal{H}^n$ and $\mathcal{H}^1$, we work with the limiting measure $\mu_\infty$, and the limiting measure $\widetilde{\mu}_\infty$ of the profile curve $l^\infty = L^\infty \cap (\mathbb{C} {\times} \{0\}^{n-1})$ respectively: 
\[ \mu_\infty(A) \, = \, \int_{L^\infty \cap A} m \, \mathcal{H}^n, \quad\quad\quad \widetilde\mu_\infty(A) \, = \, \int_{l^\infty \cap A} m \, \mathcal{H}^1. \]
From now on, we assume that $\frac{x_i}{|x_i|} \rightarrow e_1 = (1,0,\ldots,0)$, since this may be achieved by passing to a subsequence and applying a rotation.
\begin{lemma}\label{lem-s1convergence1}
	Assume that the $L^i$ are as above, converging to $\mu_\infty$ as Radon measures. If $|x_i| \rightarrow \infty$ and $\frac{x_i}{|x_i|} \rightarrow e_1$, then $\mu_\infty$ is supported on $ \mathbb{C}{\times}\mathbb{R}^{n-1} \subset \mathbb{C} {\times} \mathbb{C}^{n-1}$. If instead $x_i$ limits to $x$, then $\mu_\infty$ is supported on $\{R_{x}(\alpha, \lambda,z): \alpha \in S^{n-2}, \, \lambda \in \mathbb{R}, \, z \in \mathbb{C}{\times}\{0\}^{n-1}\}$. 
\end{lemma}
\begin{proof}
	Note that $L^i$ is supported on \[\{R_{x_i}(\alpha, \lambda,z): \alpha \in S^{n-2}, \, \lambda \in \mathbb{R},\, z \in \mathbb{C}{\times}\{0\}^{n-1} \},\] since the profile curve determines the entire submanifold. Therefore in the $|x_i|$ unbounded case, for any open set $U$ disjoint from $\mathbb{C}{\times}\mathbb{R}^{n-1} \, = \, \{T_{1}(\alpha, \lambda, z): \alpha \in S^{n-2}, \, \lambda \in \mathbb{R}, \, z \in \mathbb{C}{\times}\{0\}^{n-1} \}$, the submanifolds $L^i$ are eventually disjoint from $U$, by the convergence of $R_{x_i}$ to $T_{1}$. It follows that $L^{\infty}$ is supported on $U^c$, and so the result follows since $U$ was arbitrary. An identical argument works if $|x_i|$ is bounded.
\end{proof}
\begin{lemma}\label{lem-s1convergence2}
	Assume that $L^i$ are as above, converging to $\mu_\infty$ as Radon measures, $|x_i| \rightarrow \infty$ and $\frac{x_i}{|x_i|} \rightarrow e_1$. Assume $L^\infty$ is a cone, smooth away from the origin, with profile curve $l^\infty$ in $\mathbb{C}$. Then, denoting the ball of radius $\delta$ in $\mathbb{C}$ by $B_{\delta}^{\mathbb{C}}$, the surface area and enclosed volume of the unit sphere $S^{n-1} \subset \mathbb{R}^{n}$ by $\omega_{n-1}$ and $V_{n}$ respectively,
	\[\frac{\widetilde{\mu}_\infty(B_{\delta}^{\mathbb{C}})}{2\delta} \, = \, \frac{\mu_\infty(B_{\delta})}{\delta^n V_n}. \]
\end{lemma}
\begin{proof}
	By Lemmas \ref{lem-s1convergence} and \ref{lem-s1convergence1}, $L^{\infty}$ and $m$ are invariant under $T_1$, and supported on $\mathbb{C}{\times}\mathbb{R}^{n-1}$. Remembering that $L^\infty$ is a cone, and applying the coarea formula:
	\begin{align*}
		\mu_\infty(B_{\delta}^{\mathbb{C}^n}){ \delta^n V_n} \, &= \, \frac{1}{\delta^n V_n} \int_{B^{n-1}_{\delta}\subset\{0\}{\times}\mathbb{R}^{n-1}} \widetilde\mu_\infty \left(B_{\sqrt{\delta^2-|\alpha|^2}}^\mathbb{C} \right) d\alpha \\
		&= \, \frac{\widetilde\mu_\infty(B_{\delta}^\mathbb{C})}{\delta^{n+1} V_n} \int_{B^{n-1}_\delta}  \sqrt{\delta^2 - |\alpha|^2} \,d\alpha \\
		&= \, \frac{\widetilde\mu_\infty(B_{\delta}^\mathbb{C})}{\delta^{n+1} V_n} \int_0^\delta  r^{n-2}\omega_{n-2} \sqrt{\delta^2 - r^2}dr\\
		&= \, \frac{\widetilde\mu_\infty(B_{\delta}^\mathbb{C})}{2\delta}.
	\end{align*}
\end{proof}
Finally, we show that the $\mathcal{H}^1$ cross-sectional measures of the profile curves $l^i$ converge to the $\widetilde\mu_\infty$ measure of the limiting profile curve $l^\infty$. In the next section, this will allow us to consider the densities of the profile curve in $\mathbb{C}$, instead of the densities of the $n$-dimensional submanifolds in $\mathbb{C}^n$.
\begin{lemma}\label{lem-h2toh1}
	Assume $L^i$ are as above, converging to $\mu_\infty$ as Radon measures,\\$|x_i| \rightarrow \infty$ and $\frac{x_i}{|x_i|} \rightarrow e_1$. Assume $L^\infty$ is a cone, smooth away from the origin, with profile curve $l^\infty$ in $\mathbb{C}$. Then denoting the profile curves by $l^i$,
	\[ \mathcal{H}^1(l^i \cap B_\delta^{\mathbb{C}}) \rightarrow \widetilde{\mu}_\infty(B_{\delta}). \]
\end{lemma}
\begin{proof}
	Define the fattened disk sets:
	\begin{align*}
		C_{\delta, \Lambda}^{\infty} \, &:= \, \{ T_{1}(\alpha, \lambda, z) : \lambda \in [-\Lambda, \Lambda], \, \alpha \in S^{n-2}, \, z \in B_\delta^{\mathbb{C}} \}, \\
		C_{\delta, \Lambda}^i &:= \{ R_{x_i}(\alpha, \lambda, z) : \lambda \in [-\Lambda, \Lambda], \, \alpha \in S^{n-2}, \, z \in B_\delta^{\mathbb{C}} \}.
	\end{align*}
	By Radon measure convergence and Lemma \ref{lem-s1convergence}, it follows that
	\[ \mathcal{H}^n(L^i \cap C_{\delta,\Lambda}^i) \rightarrow \mu_\infty(C_{\delta,\Lambda}^\infty).\]
	But also, by rotation invariance and the co-area formula, denoting by $A_{n-1}(r,\lambda)$ the $(n-1)$-dimensional Hausdorff measure of the cap of $S_{r}^{n-1}$ with polar angle of $\lambda$:
	\begin{align*}
		\lim_{i \rightarrow \infty} \mathcal{H}^n(L^i \cap C_{\delta,\Lambda}^i) \, &= \, \lim_{i \rightarrow \infty} \int_{l_i\cap B_{\delta}^{\mathbb{C}}} A_{n-1} \left(|x-x_i|, \frac{\Lambda}{|x-x_i|} \right) d\mathcal{H}^1 \\
		&\leq  \lim_{i \rightarrow \infty} \left( A_{n-1} \left(|x_i|+\delta, \frac{\Lambda}{|x_i|-\delta} \right)\int_{l_i\cap B_{\delta}^{\mathbb{C}}} 1 \, \, d\mathcal{H}^1 \right) \\
		&= \, \Lambda^{n-1} V_{n-1} \lim_{i \rightarrow \infty} \left( \mathcal{H}^1 (l^i \cap B_{\delta}^{\mathbb{C}}) \right),
	\end{align*}
	and an identical inequality holds in the other direction by changing the sign of $\delta$ in the $|x_i|+\delta$, $|x_i|-\delta$ terms. Since $L^\infty$ is invariant under $T_1$, it follows that 
	\begin{align*}
		\lim_{i \rightarrow \infty} \mathcal{H}^1 (l^i \cap B_\delta^\mathbb{C}) \, = \, \frac{\mu_\infty(C_{\delta, \Lambda}^\infty)}{\Lambda^{n-1}V_{n-1}} \, = \, \widetilde{\mu}_\infty(B_\delta^{\mathbb{C}}).
	\end{align*}
\end{proof}

\section{$O(n)$-Equivariant Mean Curvature Flow in $\mathbb{C}^n$}\label{sec-s1mcf}

We now consider an almost-calibrated, connected $O(n)$-equivariant mean curvature flow in $\mathbb{C}^n$, which we denote $L_t$, with profile curve $l_t$ in $\mathbb{C}$. By Lemmas \ref{lem-noncompact} and \ref{lem-embedded}, $l_t$ is embedded and non-compact. We will denote the abstract manifold by $L$, and the Lagrangian angle at time $t$ by $\theta_t$. 
We use $s \in \mathbb{R}$ to denote the parameter along $l_t$ (note that by Lemma \ref{lem-noncompact}, $l_t$ is non-compact) and $\alpha$ for the spherical parameter, so that the Lagrangian submanifold is parametrised by
\[ L_t(s, \alpha)\, = \, (a_t(s) \, \alpha, \, \, b_t(s) \, \alpha ). \]

We will assume throughout that our flow has \textbf{planar asymptotics}. By this, we mean that our profile curve $l_t$ is graphical over finitely many lines outside of some ball $B_R$, and the graph function converges smoothly to 0 at infinity. In fact due to the equivariance, it must be graphical over exactly one or two lines, depending on whether $l_t$ passes through the origin or not. This assumption provides uniformly bounded area ratios, which are necessary to use Neves' Theorem \ref{thm-a} and \ref{thm-b}. Note that if the profile curve is asymptotic to two different lines, then for the curve to be almost-calibrated the angle between these lines must be less than or equal to $\frac{2\pi}{n}$. This follows from (\ref{eq-langle}) by considering the value of $\theta$ as the profile curve decays to the asymptotes, and will be proven more rigorously in Lemma \ref{lem-cone}.

The curvature tensors are simpler in the equivariant case. Working with  $L_t \cap (\mathbb{C}\times\{0\}^{n-1})$ for simplicity, the mean curvature vector of the Lagrangian can be divided into two key components:
\[ \vec{H} = \vec{k} - (n-1)\vec{p}, \]
where $\vec{k} = \frac{{\gamma''}^\perp}{|\gamma'|^2}$ is the curvature of the profile curve, and $\vec{p} = \frac{\gamma^\perp}{|\gamma|^2}$ is the curvature induced by the equivariance. Denoting by $\nu = \frac{J\dot\gamma}{|\dot\gamma|}$ the unit normal vector within $\mathbb{C}\times\{0\}^{n-1}$, we can also define the scalar quantities $\vec{k} = k\nu$ and $\vec{p} = p\nu$.

\subsection{Evolution Equations for Equivariant Lagrangians}
The equivariant condition significantly simplifies the study of mean curvature flow, since we may study the flow of the profile curve, given by
\begin{equation}\label{eqn-sff}
	\frac{\p l}{\p t} \, = \, \vec k - (n-1)\vec p.
\end{equation}
We refer to this as the \textbf{equivariant flow}. This is simpler than Lagrangian mean curvature flow in general as it is a codimension 1 flow, and therefore essentially a single PDE as opposed to a system. 

There are several parametrisations of the profile curve that will come in useful. Firstly, if $u:\mathbb{R} \rightarrow \mathbb{R}$ is a graph function such that our flow may be expressed as $l_t(x) = (x, u_t(x))$, then the evolution of $u_t$ under mean curvature flow is given by
\begin{align}
	\frac{\p u}{\p t} \, &= \, \frac{ u''}{1 + (u')^2} + (n-1)\frac{xu' - u}{x^2 + u^2} \label{eqn-equiflow1}\\
	&= a(u')u'' \, + \, b(x,u,u'). \label{eqn-equiflow2}
\end{align}
It is also often useful to parametrise in polar coordinates, $l_t(s) = r_t(s)e^{is},$ in which case the evolution of $r_t$ under mean curvature flow is given by
\begin{align}
	\frac{\p r}{\p t} \, &= \, -\frac{\theta'}{r}, \label{eqn-rflow}
\end{align}
where $\theta'$ is the derivative of the Lagrangian angle with respect to the angle $s$.

\subsection{Embeddedness and Avoidance Principle}

We have already seen in Lemma \ref{lem-embedded} that almost-calibrated equivariant Lagrangians in $\mathbb{C}^n$ must be embedded. In this section, we derive some more general results about embeddedness and avoidance for equivariant Lagrangian mean curvature flow, without requiring the almost-calibrated hypothesis. This will be useful later so that we can use barriers to control the flow. Though embeddedness does not typically hold for higher codimension MCF, it will for the equivariant case, since the profile curve is a codimension 1 flow.

\begin{lemma}\label{lem-emb-2}
	Let $L_t$ be a connected $O(n)$-equivariant mean curvature flow in $\mathbb{C}^n$, and assume that $T_{emb} < \infty$. Then if $T_{emb} \neq T_{sing}$, the embeddedness cannot break at the origin, i.e. there do not exist $a, b \in L$ such that $L_{T_{emb}}(a) = L_{T_{emb}}(b) = O$.
\end{lemma}

\begin{proof}
	We work with the profile curve for simplicity. Assume that there exist $a, b$ such that \[ l_{T_{emb}}(a) = l_{T_{emb}}(b) = O, \] therefore $ l_{t}(a),  l_{t}(b) \rightarrow O$ as $t \rightarrow T_{emb}$. Choose a sequence $t_n \rightarrow T_{emb}$, $t_n < T_{emb}$.
	
	\emph{Claim.} If $T_{emb} \neq T_{sing}$, then there exists $N \in \mathbb{N}$ such that for all $n \geq N$, there exists $\veps_n$ such that $l_{t_n}(a), l_{t_n}(b)$ lie in different connected components of $B_{\veps_n}(0)\cap l_{t_n}$.
	
	\emph{Proof of claim.} We prove the contrapositive, so assume that there exists a subsequence $n_k$ such that $l_{t_{n_k}}(a), l_{t_{n_k}}(b)$ lie in the same connected component of $B_\veps(0)\cap l_{t_{n_k}}$ for all suitable $\veps$. Therefore, 
	\[ \forall x \in [a,b], \quad l_{t_{n_k}}(x) \rightarrow O \]
	and so $l_{T_{emb}}$ is not immersed. Therefore $T_{emb} = T_{sing}$, and the claim is proven.
	
	We may therefore find sequences of numbers $\veps_n$ and of connected components $\alpha_n$, $\beta_n$ in $B_{\veps_n} \cap l_{t_n}$ such that $l_{t_n}(a) \in \alpha_n$, $l_{t_n}(b) \in \beta_n$. 
	If there exists a time $t$ and a point $c$ with $l_t(c) = O$, then by the equivariant symmetry and uniqueness of the mean curvature flow, for all times $s>t$, $l_s(c) = O$. It follows that if there exist times $t<s < T_{emb}$ and points $c \neq d$ with $l_t(c) = l_s(d) = O$, then
	\[ l_{s}(c)=l_{s}(d) = O,\]
	contradicting $T_{emb}$ being the first time of non-embeddedness. Therefore (up to relabelling) at least one of the sequences of connected components $\alpha_n$, $\beta_n$ never includes the origin. Without loss of generality let it be $\alpha_n$.
	
	Now let $a_n \in \mathbb{R}$ be the point such that $p_n := l_{t_n}(a_n)$ is the closest point in $\alpha_n$ to the origin; note that $p_n \rightarrow 0$. Then 
	\begin{align*}
		\left\langle \frac{\p l_{t_n}}{\p s}(a_n), p_n \right\rangle \, = \, 0 \quad &\implies \quad \left\langle p_n, \nu(a_n) \right\rangle \, = \, |p_n|,
	\end{align*}
	where $\nu$ is the outward normal, so at the space-time point $(p_n,t_n)$,
	\begin{align*}
		|A|^2 \, &= \, k^2 + 3(n-1)p^2 = k^2 + \frac{3(n-1)}{|p_n|^2}.
	\end{align*}
	This diverges to infinity as $n \rightarrow \infty$, and therefore, $T_{emb} = T_{sing}$.
\end{proof}

\begin{theorem}[Preservation of Embeddedness/Avoidance Principle]\label{thm-emb}
	Let $L_t$ be a connected $O(n)$-equivariant mean curvature flow in $\mathbb{C}^n$ with planar asymptotics, such that each end is asymptotic to a different $n$-plane. Assume $T_{emb} < \infty$. Then $T_{emb} = T_{sing}$.
	
	Additionally, if $L_t$ and $\overline{L}_t$ are two such flows, initially disjoint and embedded and with different asymptotes, then they remain disjoint until a singularity occurs.
\end{theorem}

\begin{proof}
	We prove only preservation of embeddedness. The avoidance principle follows precisely the same argument, noting that the first point of contact for the flows cannot be at the origin else the curvature would blow up at that time. Assume that $T_{emb} < T_{sing}$, for a contradiction. Then we may take a sequence of points $(x_n, t_n)$ and points $a_n, b_n \in \mathbb{R}$ such that $l_{t_n}(a_n) = l_{t_n}(b_n) = p_n$, where $t_n$ is a decreasing sequence converging to $T_{emb}$. Since the ends of the profile curve have different asymptotes, there exists $R$ such that $p_n \in B_R$ for all $n$, so passing to a subsequence there exist limits $a_n \rightarrow a$, $b_n \rightarrow b$, $p_n \rightarrow p$ such that $l_{T_{emb}}(a) = l_{T_{emb}}(b) = p$. By Lemma \ref{lem-emb-2}, $p$ is not the origin.
	
	Since $p$ is the first point of contact, we must have (at $T_{emb}$) $l'(a) = l'(b)$, and so there is a unique line $\Lambda$ through the origin parallel to the shared tangent space to $l$ at $p$. Additionally we may take $\veps$ sufficiently small such that $B_{\veps}(p) \cap l$ has two connected components for $t<T_{emb}$, which may be written as graphs $u_1, u_2$ over $\Lambda$, with $u_1 \geq u_2$ and $u_1=u_2$ at a point $x \in \Lambda$. These graphs both satisfy the equivariant mean curvature flow equation (\ref{eqn-equiflow2}).
	
	We show that the difference $v := u_1 - u_2$ also satisfies a parabolic differential equation. Defining $u_s := su_1 + (1-s)u_2$, we interpolate between the equations:
	\begin{align*}
		\frac{\p v}{\p t} \, &= \, \frac{\p u_1}{\p t} - \frac{\p u_2}{\p t} \,
		= \, a(u_1')u_1'' \, + \, b(x, u_1, u_1') \, - \, a(u_2')u_2'' \, - \, b(x, u_2, u_2')\\
		&= \, \int_0^1 \frac{\p}{\p s} \left( a(u_s')u_s'' \, + \, b(x, u_s, u_s') \right) ds \\
		&= \, \left( \int_0^1 a(u_s') ds \right) v'' \, + \, \widetilde{b}(x) v' \, + \, \widetilde{c}(x) v.
	\end{align*}
	We may therefore apply the parabolic maximum principle to conclude that $v = 0$ at some earlier time, contradicting the definition of $T_{emb}$.
\end{proof}

\subsection{The Type I Blowup}

We now return to almost-calibrated flows, and examine the Type I blowup. By Neves' Theorem A (Theorem \ref{thm-a}), any Type I blowup of our LMCF must be a union of equivariant special Lagrangian planes, and due to Theorem \ref{thm-ACLMCFsing}, almost-calibrated LMCF cannot develop Type I singularities. Therefore we expect the Type I blowup to consist of a union of multiple equivariant planes through the origin. We will show in this section that in fact it must be a pair of planes, with the same Lagrangian angle.

Throughout we will use the notation $L^i_s$ for a sequence of Type I rescalings, with factors $\lambda_i$, and profile curves $l_s^i$. As before, we assume that $L_0$ is asymptotically planar, and this implies the area bound
\[\mathcal{H}^n(L_0 \cap B_R(0)) \, \leq \, C_0R^n. \]
This implies uniformly bounded area ratios for all time by Huisken's monotonicity formula, see for example \cite{neves:1}.

The following main lemma proves that blowup sequences centered away from the centre of rotation converge to a single plane. This will be used to rule out singularities away from the origin, as well as double density planes for singularities at the origin. Throughout, we use the notation $B_{a}(x)$ for a ball of radius $a$ centered at $x$, and $B_{a}$ as shorthand for $B_a(O)$.

\begin{lemma}\label{lem-main}
	Let $L^i$ be a sequence of uniformly almost-calibrated and connected Lagrangian submanifolds in $\mathbb{C}^n$, with the property that $L^i - x_ie_1$ is an $O(n)$-equivariant submanifold of $\mathbb{C}^n$ for a sequence $x_i \in \mathbb{C}$ and $e_1 = (1,0,\ldots, 0) \in \mathbb{C}^n$. Assume that $x_i$ eventually lies outside of $B_d$ for some $d$.
	
	Assume further that the conclusions to Theorem \ref{thm-a} and \ref{thm-b} hold locally in $B_1$ for the sequence $L^i$. Explicitly,
	\begin{itemize}
		\item There exists a finite set $\{\overline{\theta}_1, \ldots, \overline{\theta}_M\}$ and integral special Lagrangian cones $\{L_1, \ldots, L_M \} $ such that for all $f \in C^2(\mathbb{R})$, $\phi \in C^\infty_c(B_1)$, 
		\begin{equation}\lim_{i \rightarrow \infty} \int_{L_i} f(\theta^i) \phi \, d\mathcal{H}^n \, = \, \sum_{j=1}^M m_j f(\overline{\theta_j}) \mu_j(\phi), \label{eq-thma}
		\end{equation}
		\item For any convergent sequence $\Sigma^i$ of connected components of $B_{2\delta}\cap L^i$ intersecting $B_{\frac{\delta}{2}}$, there exists a special Lagrangian cone $L$ with Lagrangian angle $\overline{\theta}$ such that for all $f \in C^2(\mathbb{R})$, $\phi \in C^\infty_c(B_1)$,
		\begin{equation}\lim_{i \rightarrow \infty} \int_{\Sigma_i} f(\theta^i) \phi \, d\mathcal{H}^n \, = \, m f(\overline{\theta})\mu(\phi).\label{eq-thmb}
		\end{equation}
	\end{itemize}
	Then there exists a single special Lagrangian plane $P$ with angle $\overline{\theta}$ and underlying Radon measure $\mu_P$ such that for all $\phi \in C^{\infty}_c(B_1)$, $f \in C^2(\mathbb{R})$:
	\[ \lim_{i \rightarrow \infty} \int_{L^i}f(\theta^i)\phi \, d\mathcal{H}^n \, = \,f(\overline{\theta}) \mu_P(\phi).\]
\end{lemma}

\begin{proof}
	The proof is by a density argument (extending a similar argument of A. Neves in \cite{neves:1}), a sketch of which is as follows. By the work of Section \ref{sec-s1} that allows us to work with the profile curve, we are done if we can prove that
	\[ \lim_{i \rightarrow \infty} \frac{\mathcal{H}^1(\gamma^i \cap B_\delta)}{2\delta} = 1, \]
	where $B_\delta := B_\delta(O)$. Since we know already that the limit is a union of planes it follows from this that it must be a single plane. We therefore wish to estimate this density ratio.
	
	Taking a sequence of connected components of $l^i$, which we label $\gamma^i$, (\ref{eq-thmb}) gives integral convergence of the Lagrangian angle in $B_\delta$, and since the centre of $O(n)$ symmetry $x_i$ is away from the origin, this implies a tight bound on the angle of $\dot\gamma^{i}$. We then use this to show the above density bound, for sufficiently small $\delta$. However we are not done, as there may be another, different sequence of connected components that can increase the total density further -- we must rule this out. 
	
	Considering two different connected components $\xi^i$ and $\eta^i$, they can either converge to the same Lagrangian angle, or a different one. If the limiting Lagrangian angle is different, then we can show that $\xi^i$ and $\eta^i$ must collide, perhaps in a larger ball, since the angles of their derivatives are tightly bounded around different values. On the other hand if the Lagrangian angle is the same, then we can show by embeddedness that there must be another connected component in between with different Lagrangian angle, causing a collision as before. This shows that there is in fact only one connected component to consider, and we are done.\\
	
	We now fill in the details. Let $B_{2\delta}$ be small enough so that for $i$ large, $x_i \notin B_{2\delta}$, and consider a sequence $\Sigma^i$ of connected components of $L^i\cap B_{2\delta}$ intersecting $B_{\frac{\delta}{2}}$, with profile curve $\gamma^i$. By (\ref{eq-thmb}), there exists a special Lagrangian cone $L^{\infty}$ with underlying Radon measure $\mu_\infty$ and an integer multiplicity $m$ such that for $\phi \in C^{\infty}_c(B_{\delta}), f \in C^2{\mathbb{R}}$:
	\begin{equation}\label{eq-radonconv}
		\lim_{i \rightarrow \infty} \int_{\Sigma^i} f(\theta^i) \phi d\mathcal{H}^n \, = \, m f(\overline{\theta})\mu_\infty(\phi).
	\end{equation}
	We first use this convergence to get a bound on $\arg(\dot\gamma)$. For $\varepsilon > 0$, define the following ``$\varepsilon$-good" and ``$\varepsilon$-bad" subsets of $\gamma^i\cap B_\delta$: 
	\begin{align*}
	S_\delta(\gamma^i) \, &:= \, \left\{ x \in \gamma^i \cap B_\delta \, \Big| \, |\theta^i(x) - \overline{\theta}(x)| \leq \veps \right\},\\
	T_\delta(\gamma^i) \, &:= \, \left\{ x \in \gamma^i \cap B_\delta \, \Big| \, |\theta^i(x) - \overline{\theta}(x)| > \veps \right\},
	\end{align*}
	note we suppress the dependence on $\veps$ for notational clarity. Then (\ref{eq-radonconv}) implies that
	\begin{equation}\label{eq-sepsilon}
		\forall \veps \,\,\, \exists N \,\,\, \mbox{s.t} \,\,\, \forall i > N,  \quad \mathcal{H}^1(T_\delta(\gamma^i)) < \veps.
	\end{equation}
	Our aim is therefore to estimate $\mathcal{H}^1(S_\delta(\gamma^i))$. Taking arguments with respect to the point $x_i$ and the $e_1$ direction, by (\ref{eq-langle}) the Lagrangian angle is given by
	\begin{equation}
		\theta \, = \, (n-1)\arg(\gamma) \, + \, \arg(\dot\gamma).
	\end{equation}
	Denoting $b_i := \arg(O - x_i)$, $b_i$ converges to some $b$ (after passing to a subsequence if necessary). Then on $B_{\delta}$ we have the bound 
	\[\arg(\gamma^i) \, \in \, \left(b_i - \sin^{-1} \left(\tfrac{\delta}{|x_i|}\right), b_i + \sin^{-1} \left(\tfrac{\delta}{|x_i|}\right)\right),\] 
	(see Figure \ref{fig-0}) and therefore on $S_\delta(\gamma^i)$, taking $i$ sufficiently large so that $|b-b_i| < \veps$, we obtain a bound on the argument of $\dot\gamma^i$:
	\begin{equation}\label{eq-rhodef}
		|\arg(\dot\gamma^i) - \overline{\theta} + (n-1)b| \, \leq \, (n-1)\sin^{-1} \left(\tfrac{\delta}{|x_i|}\right) + 2\veps =: \rho(\delta, \veps).
	\end{equation}	
	\begin{figure}
		\centering
		\includegraphics[width=0.55\linewidth]{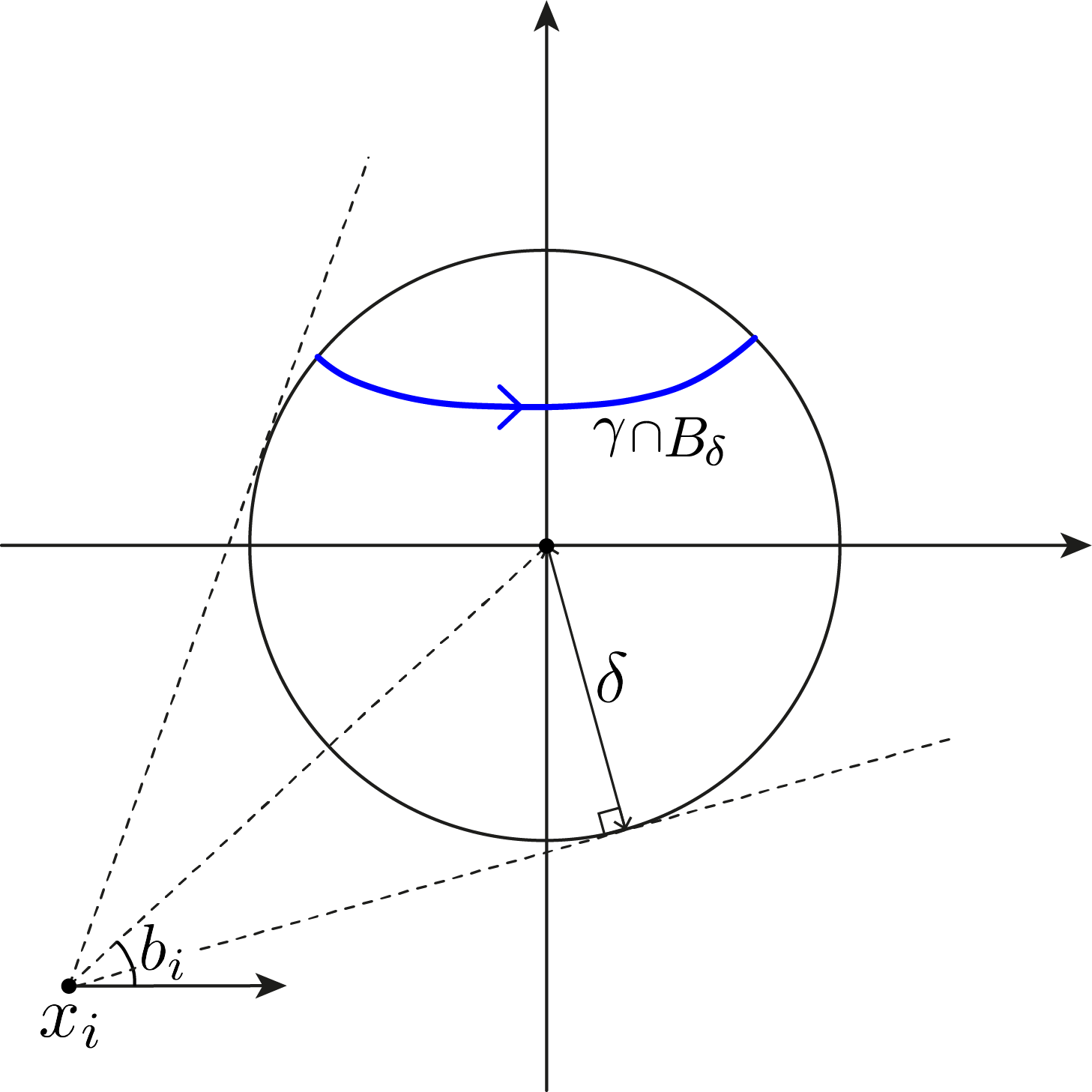}
		\caption[Diagram for the proof of Lemma \ref{lem-main}]{The setup of Lemma \ref{lem-main}.}
		\label{fig-0}
	\end{figure}
	Parametrise by unit speed, so that $\dot\gamma^i(s) \, = \, e^{i(\lambda(s)+\overline{\theta}-(n-1)b)}$ for an angle function $\lambda(s)$. Then equation (\ref{eq-rhodef}) implies $|\lambda(s)| \, \leq \, \rho $, and therefore
	\begin{align}
		\left| \int_{S_\delta(\gamma^i)} \dot\gamma^i(s) \, ds \right| \, &\geq \, \left| \int_{S_\delta(\gamma^i)} e^{i(\overline{\theta} - (n-1)b)}\cos(\lambda(s)) \, ds \right| \geq \, \mathcal{H}^1(S_\delta(\gamma^i))\cos\rho. \label{eq-h1bound}
	\end{align}
	We'd like to use (\ref{eq-h1bound}) to bound $\mathcal{H}^1(S_\delta(\gamma^i))$, so we need to bound $ \left| \int_{S_\delta(\gamma^i)} \dot\gamma^i(s) \, ds \right|$. If $\gamma^i \cap B_\delta$ was a single connected component, this would be simple, as the integral of $\dot\gamma^i$ over $B_\delta$ would then be less than $2\delta$ by the fundamental theorem of calculus. However there may be more connected components to worry about. The following lemma demonstrates that if we widen our ball slightly, we will only have to worry about one connected component. After its proof, we will resume the proof of Lemma \ref{lem-main}.
	
	\begin{lemma}\label{lem-concomp}
		Assume that we have the setup of Lemma \ref{lem-main}. Then for sufficiently small $\delta, \veps$, there exists $N$ such that for all $i>N$, there is only one connected component of $l^i \cap B_{\delta + 3\veps}$ intersecting $B_\delta$.
	\end{lemma}
	\begin{proof}
		We demonstrate that for sufficiently small $\delta$, $\varepsilon$, there exists $N$ such that for all $i>N$:
		\begin{itemize}
			\item[1.] Two distinct sequences of connected components of $l^i \cap B_{\delta + 3\veps}$ intersecting $B_\delta$ can't have different Lagrangian angles in the limit.
			\item[2.] If two distinct sequences of connected components of $l^i \cap B_{\delta + 3\veps}$ intersecting $B_\delta$ have the \textit{same} limiting Lagrangian angle, we can find a third connected component $\zeta^i$ with a different limiting Lagrangian angle.
		\end{itemize}
		Together, these two claims complete the proof.
		
		\textit{Proof of 1.} By (\ref{eq-thma}), there are a finite number of possible limiting Lagrangian angles for these curves, $\{\overline{\theta}_1, \ldots, \overline{\theta}_M\}$. These correspond bijectively to a finite number of possible limiting values for the argument of the tangent vector, $\arg(\dot \gamma^i)$ (see (\ref{eq-rhodef})):
		\[ A = \{\alpha_1, \ldots, \alpha_M\} \, := \, \{\overline{\theta}_1-(n-1)b, \ldots,\overline{\theta}_M-(n-1)b \}.  \] By the almost-calibrated condition, these angles are all different modulo $\pi$, and so any two straight lines representing different angles in $A$ that intersect $B_\delta$ must intersect in a sufficiently large ball. We may therefore choose $R$ large enough such that any two curves $\eta$ and $\xi$ in $B_{R\delta}$ intersecting $B_{\delta}$ such that $\arg(\dot \eta)$ and $\arg(\dot \xi)$ are $\veps$-close to distinct values in $A$ (outside a set of $\mathcal{H}^1$-measure $\veps$) must collide inside $B_{R\delta}$.
		
		Now for a contradiction, assume that, after passing to a subsequence, for all $i$ there exist two distinct connected components $\eta^i$ and $\xi^i$ of $l^i \cap B_{\delta + 3\veps}$ intersecting $B_\delta$ whose Lagrangian angles converge to distinct values $\overline{\theta}_\eta$ and $\overline{\theta}_\xi$. Now extend $\eta^i$ and $\xi^i$ to the connected components in $B_{R\delta}$ that contain them (which may be the same): call these $\overline{\eta}^i$ and $\overline{\xi}^i$. For sufficiently small $\delta$ we can apply the same argument as in the proof (so far) of Lemma \ref{lem-main} and show that, for sufficiently large i, (\ref{eq-rhodef}) holds for $\overline{\eta}^i$ and $\overline{\xi}^i$ in $B_{R\delta}$ outside a set of $\mathcal{H}^1$-measure $\veps$, with the Lagrangian angles $\overline\theta_\eta$ and $\overline\theta_\xi$ respectively. This implies that the connected components must be distinct, but by the choice of $R$, $\overline{\eta}^i$ and $\overline{\xi}^i$ must then collide for $i$ sufficiently large, contradicting embeddedness.
		
		\textit{Proof of 2.} Assume that (after passing to a subsequence) for all $i$ there exist two distinct connected components $\eta^i$ and $\xi^i$ of $l^i \cap B_{\delta + 3\veps}$ intersecting $B_\delta$, and that the Lagrangian angles of $\xi$, $\eta$ converge to the same value $\overline{\theta}$; without loss of generality we assume that $\overline{\theta} - (n-1)b = 0$. 
		
		We first show that $\xi^i$ must enter the ball $B_{\delta + 3\veps}$ on the left-hand side and leave on the right-hand side. Work with a unit-speed parametrisation, $\dot\xi^i(s) = e^{i\lambda(s)}$. Since $\xi^i$ intersects $B_{\delta}$, there is some $s_0$ such that $\xi^i(s_0) = p \in B_{\delta}$. By connectedness,
		\[ \mathcal{H}^1(\{ \xi(s) : s \geq s_0 \}) \, \geq 3\veps. \]
		Therefore by splitting the set into $S_{\delta + 3\veps}$ and $T_{\delta + 3\veps}$ we can calculate the horizontal and vertical distance travelled:
		\begin{align*}
			\int_{\{s \geq s_0\}} \cos(\lambda(s)) \, ds \, &= \, \int_{\{s \geq s_0\} \cap T_{\delta + 3\veps}} \cos(\lambda(s)) \, ds \, + \, \int_{\{s \geq s_0\} \cap S_{\delta + 3\veps}} \cos(\lambda(s)) \, ds  \\
			&\geq \, -\veps + 3\veps\cos(\rho) \geq \tfrac{\veps}{2},\\
			\int_{\{s \geq s_0\}} |\sin(\lambda(s))| \, ds \, &= \, \int_{\{s \geq s_0\} \cap T_{\delta + 3\veps}} |\sin(\lambda(s))| \, ds \, + \, \int_{\{s \geq s_0\} \cap S_{\delta + 3\veps}} |\sin(\lambda(s))| \, ds  \\
			&\leq \, \veps + 3\veps\sin(\rho) \leq 2\veps,
		\end{align*}
		by (\ref{eq-rhodef}) and (\ref{eq-sepsilon}), since by taking $\delta, \varepsilon$ sufficiently small we may make $\rho(\delta,\varepsilon)$ as close to 0 as we like. This shows that $\xi^i$ leaves the ball on the right-hand side, since $p_0$ must be to the left of the exit point, and less than $2\veps$ vertically separated from it. An identical argument shows that $\xi^i$ enters on the left-hand side. The same is true for $\eta^i$.
		
		Now if these were the only connected components, we have the situation depicted in Figure \ref{fig-1}. Since $L^i$ is connected, either $A_R$ joins to $C_L$ or $C_R$ joins to $A_L$. In both situations, one end of the curve must be trapped in a compact region of the plane by embeddedness (the $C_R$ or $A_L$ end in the former case, the $C_L$ or $A_R$ end in the latter), which is a contradiction.
		\begin{figure}
			\centering
			\includegraphics[width=0.4\linewidth]{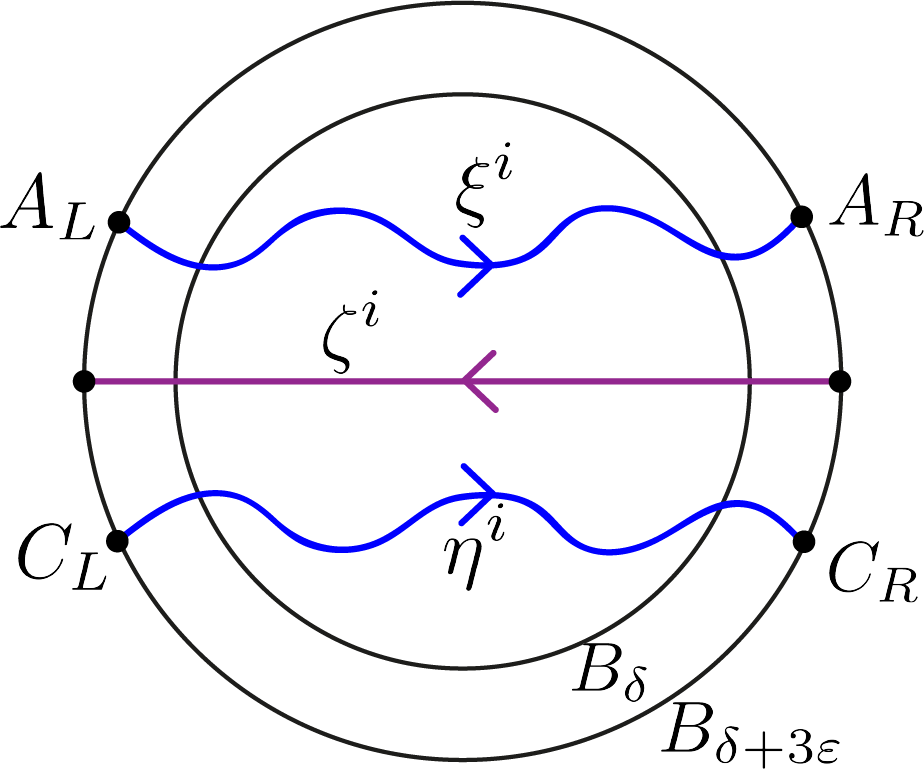}
			\caption[Diagram for the proof of Lemma \ref{lem-concomp}]{Two connected components with the same angle must have another between them.}
			\label{fig-1}
		\end{figure}
		Therefore there must be another connected component $\zeta^i$ in $B_{\delta + 3\veps}$; to solve the above problem it must be a curve from right to left, in the middle of $\xi^i$ and $\eta^i$ (see Figure \ref{fig-1}). By the above argument, since $\zeta^i$ does not enter on the left and leave on the right it must have a different limiting Lagrangian angle, and so the argument is complete.
	\end{proof}
	We now resume the proof of Lemma \ref{lem-main}. Taking $\delta$, $\veps$ sufficiently small, we know by Lemma \ref{lem-concomp} that for sufficiently large $i$, there is only one connected component of $\gamma^i \cap B_{\delta+3\veps}$ that intersects $B_\delta$: call it $\tilde\gamma^i$. Also for sufficiently large $i$, $\mathcal{H}^1(T_{\delta+3\veps}(\tilde\gamma^i)) < \veps$. Using this, (\ref{eq-rhodef}) and (\ref{eq-h1bound}), we estimate for sufficiently large $i$ using a unit-speed parametrisation (suppressing the superscript $i$ and defining $\tilde{\rho}:= \rho(\delta + 3\veps, \veps)$, $\rho := \rho(\delta, \veps)$ for readability):
	\begin{align*}
		\mathcal{H}^1(S_{\delta + 3\veps}(\tilde\gamma) ) \cos\tilde{\rho} \, &\leq \, \left| \int_{S_{\delta + 3\veps}(\tilde\gamma)} \dot{\tilde\gamma}\, ds \right| \, \leq \, \left| \int_{B_{\delta + 3\veps}\cap \tilde\gamma} \dot{\tilde\gamma}\, ds \right| \, + \, \veps \, \leq \, 2\delta + 7 \veps\\
		\implies \mathcal{H}^1(S_{\delta}(\gamma)) \, &\leq \, \mathcal{H}^1(S_{\delta + 3\veps}(\tilde\gamma)) \, \leq \, \frac{2\delta + 7\veps}{\cos\tilde{\rho}}.
	\end{align*}
	Finally, using this and (\ref{eq-sepsilon}) we can estimate our density ratio:
	\begin{align*}
		\frac{\mathcal{H}^1(\gamma \cap B_\delta)}{2\delta} \, &= \, \frac{\mathcal{H}^1(S_\delta(\gamma))}{2\delta} \, + \, \frac{\mathcal{H}^1(T_\delta(\gamma))}{2\delta} \, \leq \, \frac{1}{\cos \tilde{\rho}} \, + \, \frac{\veps}{\delta}\left(\frac{7}{2\cos \tilde{\rho}} + \frac{1}{2} \right).
	\end{align*}
	By (\ref{eq-rhodef}), $\cos \tilde{\rho} = \cos (\rho(\delta + 2\veps, \veps)) = 1 + O(\delta,\veps)$. Therefore, taking $\delta$ and $\tfrac{\veps}{\delta}$ sufficiently small ensures that the density ratio is bounded away from 2. By Lemmas \ref{lem-speciallagcones} and \ref{lem-h2toh1}, we must have that this density ratio converges to an integer, which due to the bound must be 1. Therefore, the limit of the sequence $\Sigma^i$ of connected components is a single plane.
	
	Finally, Lemma \ref{lem-concomp} implies that there are no other connected components of $l^i\cap B_{2\delta}$ intersecting $B_{\frac{\delta}{2}}$, so we have in fact proven that $L^i$ converges to a single Lagrangian plane.
\end{proof}
Now we apply Lemma \ref{lem-main} to get our main results.
\begin{theorem}\label{thm-singorigin}
	Let $L_t$ be an almost-calibrated, connected $O(n)$-equivariant mean curvature flow in $\mathbb{C}^n$ with planar asymptotics. Then if $L_t$ has a finite-time singularity, it must occur at the origin.
\end{theorem}

\begin{proof}
	Assume for a contradiction that such a flow has a singularity away from the origin. Without loss of generality, it is at a point $(ae^{ib},0,\ldots,0) \in \mathbb{C}{\times}\{0\}^{n-1}$, since otherwise we may just perform a rotation that leaves the flow unaffected. Note that the planar asymptotics imply uniformly bounded area ratios, and by Lemma \ref{thm-s1lag} the flow is rational. Taking a sequence of rescalings $L^j_s$ around $ae^{ib}$ with factor $\lambda^j$, the conclusions to Theorems \ref{thm-a} and \ref{thm-b} therefore hold for almost all $s$. The centre of rotation for $L^j_s$ is $x_j := -\lambda_jae^{ib}$, whose size diverges to $\infty$. 
	
	We may therefore apply Lemma \ref{lem-main} to conclude that $L^i_s$ converges to a density 1 Lagrangian plane for almost all $s$. This convergence is smooth by White regularity (Theorem \ref{thm-white}), by the following argument. Choosing a space-time point $\overline{X} = (\overline{x},\overline{s})$, for any $X \in P_r(\overline{X})$ we have by Huisken monotonicity (Theorem \ref{thm-monotonicity}):
	\begin{align*}
		\lim_{i\rightarrow \infty}\Theta(L^i,X,r) \, &= \, \lim_{i \rightarrow \infty}\int_{L^i_{s-r^2}} \Phi_X(x,s-r^2) \, d\mathcal{H}^n\,\\
		&\leq \, \lim_{i \rightarrow \infty}\int_{L^i_{\overline{s}-2r^2}} \Phi_X(x,\overline{s}-2r^2) \, d\mathcal{H}^n \leq 1,
	\end{align*} 
	for any $r$ such that $L^i_{\overline{s}-2r^2}$ converges to a density 1 Lagrangian plane. The last inequality holds since $\Phi$ integrates to 1 over a plane including $X$, and less than $1$ on any other plane. It follows that by White regularity that the curvatures are all bounded, giving smooth convergence upon passing to a subsequence by Arzel\`a-Ascoli.
	
	But since the singularity is Type II by Theorem \ref{thm-ACLMCFsing} we should have that the curvature of these rescalings is diverging. So, we have a contradiction.
\end{proof}
Next, we prove the uniqueness of the Type I blowup. We will need the following lemma, which gives bounds on the argument of $l_t$.
\begin{figure}
	\centering
	\includegraphics[scale=0.8]{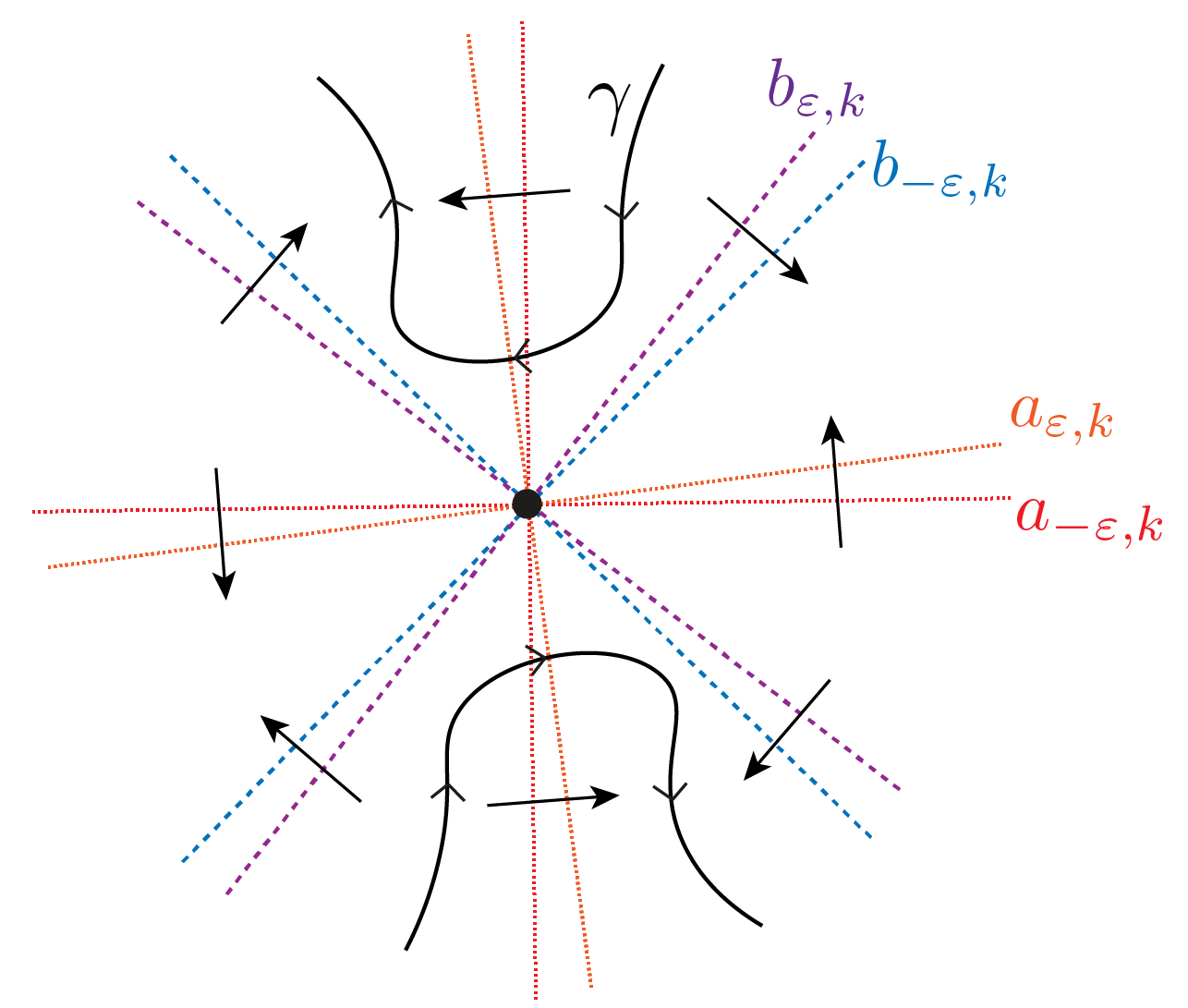}
	\caption[Diagram for the proof of Lemma \ref{lem-cone}]{The half-lines $a_{\varepsilon,k}$, $a_{-\varepsilon,k}$, $b_{\varepsilon,k}$ and $b_{-\varepsilon,k}$ in the proof of Lemma \ref{lem-cone}, in the case where $n=4$.}
	\label{fig-conelemma}
\end{figure}
\begin{lemma} \label{lem-cone}
	Let $L$ be a connected, $O(n)$-equivariant Lagrangian submanifold, with planar asymptotics. Assume that the profile curve $l$ does not contain the origin, and that $L$ is almost-calibrated; explicitly that there exist $\overline{\theta}$ and $\veps$ such that
	\[ \theta \in \left( \overline{\theta} - \frac{\pi}{2} + \veps, \, \overline{\theta} + \frac{\pi}{2} - \veps \right). \]
	Then for a connected component $\gamma$ of $l$, there exists a cone of angular width strictly less than $\frac{2\pi}{n}$ containing $\gamma$.
\end{lemma}
\noindent\textit{Proof.}
Consider the following half-lines, for $k \in \mathbb{N}$ (see Figure \ref{fig-conelemma} for a diagram):
\begin{align*}
	\arg(a_{\varepsilon, k}) \, &= \, \tfrac{\overline\theta}{n}\,+\,\tfrac{2\pi k}{n} \, + \, \tfrac{\pi}{2n} \, + \, \tfrac{\varepsilon}{n},\\
	\arg(a_{-\varepsilon, k}) \, &= \, \tfrac{\overline\theta}{n}\,+\,\tfrac{2\pi k}{n} \, + \, \tfrac{\pi}{2n} \, - \, \tfrac{\varepsilon}{n},\\
	\arg(b_{\varepsilon, k}) \, &= \,\tfrac{\overline\theta}{n}\,+\,\tfrac{2\pi k}{n} \, - \, \tfrac{\pi}{2n} \, + \, \tfrac{\varepsilon}{n},\\
	\arg(b_{-\varepsilon, k}) \, &= \,\tfrac{\overline\theta}{n}\,+\,\tfrac{2\pi k}{n} \, - \, \tfrac{\pi}{2n} \, - \, \tfrac{\varepsilon}{n}.
\end{align*}
By the almost-calibrated condition, it can be shown that the curve may only pass through the lines $a_{\varepsilon,k}$ and $a_{-\varepsilon,k}$ in a clockwise direction, and the lines $b_{\varepsilon,k}$ and $b_{-\varepsilon,k}$ in an anticlockwise direction. For example, for a contradiction assume that $\gamma$ passes through $a_{{\varepsilon,k}}$ anticlockwise. Then at that point, by (\ref{eqn-laganglecalc}),
\begin{align*}
	\theta \, &\in \, \left( n\arg(a_{\varepsilon,k}),\,\,n\arg(a_{\varepsilon,k}) + \pi \right) \mod 2\pi\\
	&\equiv \, \left( \overline\theta + \tfrac{\pi}{2}  + \varepsilon,\,\, \overline\theta + \tfrac{3\pi}{2}  + \varepsilon \right) \mod 2\pi,
\end{align*}
which contradicts the almost-calibrated condition.

But now it is clear that the curve must remain in a cone of angle less than $\frac{2\pi}{n}$, as it is constrained by the above lines.
\qed

\begin{theorem}\label{thm-typei}
	Let $L_t$ be an almost-calibrated, connected $O(n)$-equivariant mean curvature flow in $\mathbb{C}^n$ with planar asymptotics.
	Then the Type I blowup of any finite-time singularity is a special Lagrangian cone consisting of a transverse pair of planes $P_1 \cup P_2$ whose profile curves span an angle of $\tfrac{\pi}{n}$, and does not depend on the sequence of rescalings.
\end{theorem}

\noindent\textit{Proof.}
We will first rule out planes with density greater than 1 in the limit, and then demonstrate that a single transverse pair of planes is the only option. We know from Theorem \ref{thm-singorigin} that the singularity must occur at the origin, therefore the centre of rotation for $L^i_s$ is $O$. We also know by Theorem \ref{thm-a} and Lemma \ref{lem-speciallagcones} that any blowup sequence $L^i_s$ converges subsequentially for almost all $s$ to a finite number of special Lagrangian cones. Fix such an $s$; we suppress the subscript for clarity. 

\begin{figure}
	\centering
	\includegraphics[width=0.5\linewidth]{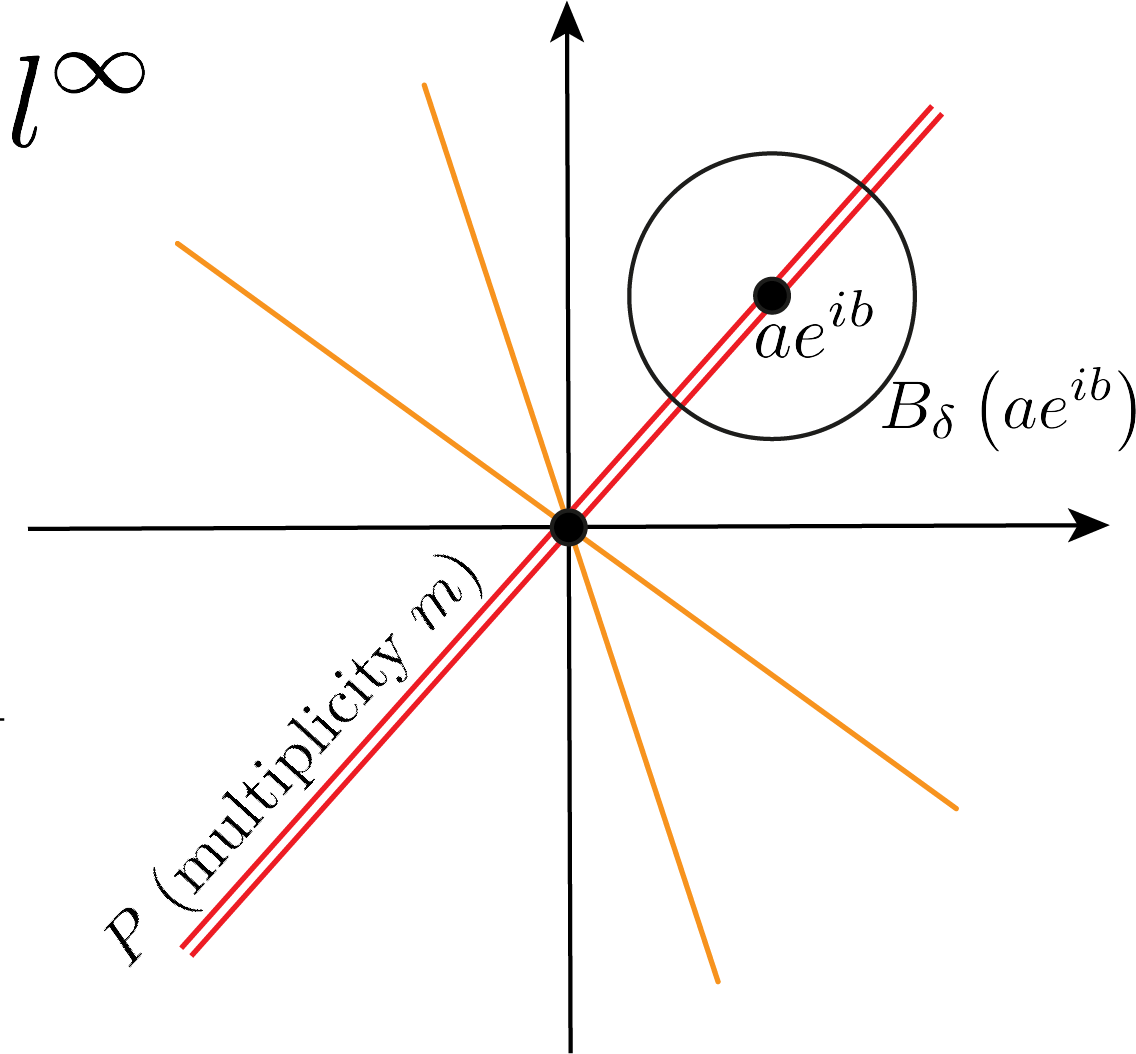}
	\caption[Diagram of a Type I blowup for the proof of Theorem \ref{thm-typei}]{The profile curve $l^{\infty}$ of a Type I blowup in the proof of Theorem \ref{thm-typei}.}
	\label{fig-2}
\end{figure}

Assume that one of the limiting planes, $P$, has a multiplicity $m > 1$. Then there is a point $ae^{ib}$ with $a < \frac{1}{4}$ and $\delta$ small enough such that all other planes in the blowup do not intersect $B_\delta (ae^{ib})$, and so
\[ \lim_{i \rightarrow \infty} \frac{\mathcal{H}^1(l^i \cap B_\delta(ae^{ib}))}{2\delta} \, \rightarrow \, m \]
(see Figure \ref{fig-2}). Now for $2\veps < \delta$, any sequence of connected components of $L^i \cap B_{2\veps}(ae^{ib})$ intersecting $B_{\frac{\veps}{2}}(ae^{ib})$ may be extended to a sequence of connected components of $B_1$ intersecting $B_{\frac{1}{4}}$. These converge to a special Lagrangian in $B_1$ by Theorem \ref{thm-b}, which must be $P$ with some multiplicity.

It therefore follows that the conclusions to Theorems \ref{thm-a} and \ref{thm-b} apply to the flows obtained by translating $ae^{ib}$ to the origin and scaling by $\frac{1}{\delta}$, locally inside the ball $B_1$. We may therefore apply Lemma \ref{lem-main} to the resulting sequence and conclude that $m=1$.

Now we show that a special Lagrangian pair of planes is the only option for the Type I blowup, working with the profile curve throughout. $\xi^i$, $\eta^i$ will denote the profile curves of sequences of different connected components of $L^i \cap B_\delta$ intersecting $B_{\frac{\delta}{4}}$. We will rule out a single line in the limit, 3 or more lines in the limit, and two separate lines coming from different connected components. This will leave the only option as a pair of lines coming from a single sequence of connected components.

\emph{One unit-density line.} Assume $\xi^i$ converges to a unit-density line; by White regularity (Theorem \ref{thm-white}) this convergence is smooth. But then there is no curvature blowup in the Type I rescalings, and by Theorem \ref{thm-ACLMCFsing} the singularity must be Type II, so this is a contradiction.

\emph{Two unit-density lines from different connected components.} Assume $\xi^i$ and $\eta^i$ converge to distinct lines. By White regularity (Theorem \ref{thm-white}) they must converge smoothly to the lines in any annulus, but this means that they must intersect each other at the origin for sufficiently large $i$ by the reflective symmetry, which is impossible since $l^i$ is embedded.

\emph{More than three unit-density lines.} By White regularity (Theorem \ref{thm-white}), we have smooth convergence to the Type I blowup in the annulus $B_\delta \setminus B_{\frac{\delta}{4}}$. Take $N$ sufficiently large, so that for $i > N$ and inside this annulus, the profile curve $l^i$ can be expressed as a graph over the limiting lines.

\begin{figure}
	\includegraphics[width=0.5\linewidth]{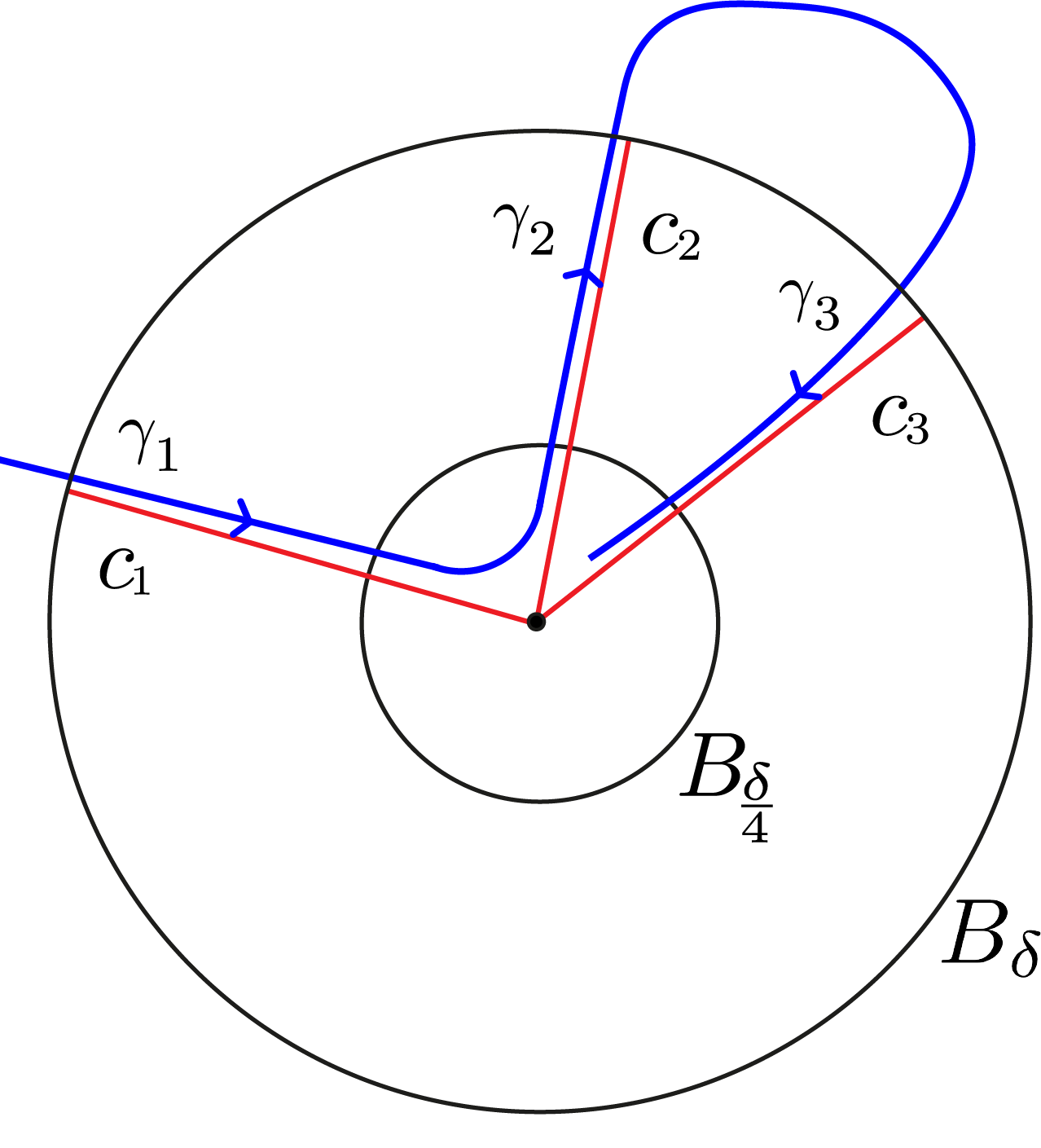}
	\centering
	\caption[Diagram of the curves $\gamma_1,\gamma_2$ and $\gamma_3$ for the proof of Theorem \ref{thm-typei}]{The curves $\gamma_1,\gamma_2$ and $\gamma_3$.}
	\label{fig-im4e}
\end{figure}

Giving $l^i$ an orientation, label the first, second and third connected components of $l^i \cap ( B_\delta {\setminus} B_{\frac{\delta}{4}} ) $ by $\gamma_1$, $\gamma_2$ and $\gamma_3$ respectively (If $l^i$ has two disconnected components, we make this definition using one half of it, $\gamma^i$). By passing to a subsequence, we may assume that these curves always lie over the same limiting half-lines; we denote the limiting half-line over which $\gamma_k$ is a graph by $c_k$, and the argument of $c_k$ by $\alpha_k$. Assume that $\gamma_2$ is clockwise from $\gamma_1$ along $l^i$ (the other case follows by an identical argument). Note that the curve $l^i$ does not pass through the origin between $\gamma_1$ and $\gamma_2$, by considering the reflective symmetry, and so the orientations of $\gamma_1$, $\gamma_2$ and $\gamma_3$ are towards, away from, and towards the origin respectively. (see Figure \ref{fig-im4e}).

Since $\gamma_1$ and $\gamma_2$ are part of the same connected component, the limiting Lagrangian angle must be the same, and since we also have the argument bound from Lemma \ref{lem-cone} it follows that
\begin{equation}
	\alpha_1 - \frac{\pi}{n} = \alpha_2. \label{eq-alpha1alpha2}
\end{equation}
Additionally, the curve $\gamma_3$ cannot be between $\gamma_1$ and $\gamma_2$. If it was, the curve $l^i$ would have to leave $B_{\delta}$ again after $\gamma_3$ between these curves by embeddedness, and since it would be part of the same connected component as $\gamma_3$, it would make an angle of $\frac{\pi}{n}$ with it in the limit. But since the angle between $c_1$ and $c_2$ is $\frac{\pi}{n}$, this would imply that $c_3 = c_1$ or $c_3 = c_2$, and we have ruled out the possibility of double-density lines. It follows that 
\[ \alpha_3 \leq \alpha_2 \leq \alpha_1. \]
By the smooth convergence, for all $\veps$ we may take $N$ large such that if $i > N$, then (keeping the orientation of the curves in mind):
\[ |\arg(\dot{\gamma_1}) - \alpha_1 + \pi| \leq \veps, \quad |\arg(\dot{\gamma_2}) - \alpha_2| \leq \veps, \quad |\arg(\dot{\gamma_3}) - \alpha_3 + \pi| \leq \veps,  \]
\[ |\arg(\gamma_k) - \alpha_k| < \veps. \]
Therefore, denoting the Lagrangian angle of $\gamma_i$ by $\theta_i$,
\begin{align*}
	\theta_1 \, &= \, \arg(\dot{\gamma}_1) + (n-1)\arg({\gamma}_1) \, \geq \, n\alpha_1 - \pi - n\veps, \\ \implies \theta_3 \, &= \,  \arg(\dot{\gamma}_3) + (n-1)\arg(\gamma_3) \, \leq \, n\alpha_3 - \pi + n\veps \, \leq \, \theta_1 - \pi + 2n\veps.
\end{align*}
Taking $\veps$ sufficiently small gives a contradiction to the almost-calibrated condition.

We therefore must have a single pair of lines in the limit, with the same Lagrangian angle $\overline{\theta}$. These lines must span an angle of $\frac{\pi}{n}$, by the same argument that gave (\ref{eq-alpha1alpha2}). 

Finally, we prove uniqueness of the Type I blowup. Assume the singularity happens at the spacetime point $(O,0)$, and consider a convergent Type I blowup sequence $L_s^{\lambda_i}$ corresponding to a time sequence $\{\lambda_i\}_{i=1}^\infty$, which we now know converges to a pair of special Lagrangian planes with a constant Lagrangian angle $\overline\theta$. 

$|\theta-\overline\theta|^2$ is a subsolution to the heat equation. Therefore, by Theorem \ref{thm-monotonicity} it follows that $\int_{L_t} |\theta-\overline\theta|^2 \Phi_{(O,0)} d\mathcal{H}^n$ is a decreasing function of $t$. Additionally, by Theorem \ref{lem-huiskenscaleinvariance},
\begin{equation} \int_{L_{-\lambda_i^{-2}}} |\theta-\overline\theta|^2 \Phi_{(O,0)} d\mathcal{H}^n \, = \, \int_{L_1^{\lambda_i}} |\theta-\overline\theta|^2 \Phi_{(O,0)} d\mathcal{H}^n. \label{eq-thetaintegralrelation}
\end{equation}
By Theorem \ref{thm-a}, this integral must converge to 0 as $i \rightarrow \infty$. Therefore by the monotonicity, $\int_{L_t} |\theta-\overline\theta|^2 \Phi_{(O,0)} d\mathcal{H}^n$ converges to zero along any subsequence; in particular by (\ref{eq-thetaintegralrelation}) any other Type I blowup must have the same Lagrangian angle $\overline\theta$.

Uniqueness of the Type I blowup now follows from the fact that there is only one such pair of lines with Lagrangian angle $\overline{\theta}$ in the cone given by Lemma \ref{lem-cone}, since this cone spans an angle of strictly less than $\frac{2\pi}{n}$. \qed

\begin{remark}
	One can show that the possible Type I blowups of almost-calibrated Lagrangian mean curvature flow form a connected subset of the space of special Lagrangian cones. The uniqueness then follows from the observation that pairs of Lagrangian planes with the same Lagrangian angle are isolated points in this space, without using Lemma \ref{lem-cone}. Similar arguments have been used in minimal surface theory to prove uniqueness of singularity models.
\end{remark}

\subsection{Finite-Time Singularity}

In this section, we analyse how the topology of the flow affects its long-time behaviour. The first possibility is that the curve passes through the origin - in this case we have long-time existence.

\begin{theorem}\label{thm-originlte}
	Let $L_t$ be an almost-calibrated, connected, $O(n)$-equivariant mean curvature flow in $\mathbb{C}^n$ with planar asymptotics. Assume that the profile curve of the initial condition, $l_0$, passes through the origin.
	Then $L_t$ exists for all time.
\end{theorem}

\begin{proof}
	Assume for a contradiction that there is a finite-time singularity. By Theorem \ref{thm-typei}, the profile curve of any Type I blowup must be a pair of lines. By White regularity, Theorem \ref{thm-white}, we must have smooth convergence of the rescalings in an annulus to this pair of lines - this creates 4 `ends' for the profile curve of each rescaling on the outer boundary of the annulus.
	
	Now in any rescaling, one connected component must go through the origin - and therefore this connected component must account for \emph{opposite} ends. The other two ends can only be joined if the curve is not embedded, and so we have a contradiction.
\end{proof}

The other option is that the profile curve doesn't pass through the origin - then $l_t$ is asymptotic to 2 lines. In this case, whether a singularity forms is dependent on the angle between the asymptotes. The following is a slight generalisation of an argument of Neves \cite{neves:1}.

\begin{lemma}
	If $\frac{2\pi}{n} > \beta > \frac{\pi}{n}$, then the Lagrangian mean curvature flow $L^n_t$ in $\mathbb{C}^n$ with profile curve $\eta_t$ starting at the initial condition with profile curve $\eta_0$ forms a singularity at the origin in finite time.
\end{lemma}

\begin{proof}
	The equivariant flow $\eta_t$ with initial condition
	\[ \eta_0(s) := \left( \sin(\tfrac{\pi s}{\beta}) \right)^{-\tfrac{\beta}{\pi}}e^{is}\]
	may be expressed in polar form, $r_t(s)e^{is}$, until a singularity forms. This can be proven using a Sturmian theorem - see \cite{neves:1} for details. 
	
	We may then look at the evolution of the area under the curve between angles $\veps$ and $\beta-\veps$, using the evolution equation (\ref{eqn-rflow}):
	\begin{align*}
	A_{\veps, t} \, &:= \, \frac{1}{2} \int_\veps^{\beta-\veps}r_t^2 ds, \quad 
	\frac{d A_{\veps,t}}{dt} \, = \, \int_\veps^{\beta-\veps} \theta_t' ds \,
	= \, \theta(\veps) - \theta(\beta-\veps).
	\end{align*}
	Using the fact that 
	\[ \theta_t(s) \, = \, \arg \left( (r'(s) + ir(s))e^{is} \right) + (n-1)\arg \left( r(s)e^{is} \right) \, = \, ns + \arg(r' + ir), \]
	it follows that $\theta_t(s) \in (ns, ns+\pi)$. Therefore, if $\pi-n\beta < 0$, we may choose $\veps$ sufficiently small such that
	\begin{align*}
	\frac{d A_{\veps,t}}{dt}\,  &< \, 2n\veps + \pi - n\beta \, < \, -C
	\end{align*}
	for a positive constant $C$. It follows that a singularity must form in finite time, and by Theorem \ref{thm-singorigin} it must occur at the origin.
\end{proof}

\begin{theorem}\label{thm-fts}
	Let $L_t$ be an almost-calibrated, connected, $O(n)$-equivariant mean curvature flow in $\mathbb{C}^n$ with planar asymptotics. Assume that the profile curve of the initial condition, $l_0$ does not pass through the origin, and that the angle $\alpha$ between the asymptotes of the profile curve is strictly between $\tfrac{\pi}{n}$ and $\tfrac{2\pi}{n}$.
	Then $L_t$ forms a finite-time singularity.
\end{theorem}

\begin{proof}
	Working with a connected component $\gamma$ of the profile curve and taking $\beta < \alpha$, Lemma \ref{lem-cone} gives us a cone that $\gamma_t$ remains in until a singularity forms. We may then find a scaled and rotated copy of Neves' curve $\eta_0$ that also lies in this cone, further away from $\gamma_0$ than the origin, that does not intersect it. By the avoidance principle for equivariant MCF, Theorem \ref{thm-emb}, under equivariant MCF these curves do not intersect until one forms a singularity. Since $\eta_t$ descends to the origin within the cone, $\gamma_t$ is also forced to the origin; here the curvature blows up and so a singularity must occur. \end{proof}

\subsection{The Type II Blowup}

In this section we analyse the Type II blowup of a singularity of our equivariant LMCF. Since by Theorem \ref{thm-originlte} an initial profile curve through the origin cannot form a finite-time singularity under MCF, we assume throughout this section that the initial profile curve $l$ avoids the origin, and therefore consists of two connected components, $\gamma$ and $-\gamma$.

We first show that any Type II blowup of an LMCF must be a special Lagrangian with the same Lagrangian angle as the Type I blowup, and then we will conclude that the only possibility for a Type II blowup is the Lawlor neck of Lemma \ref{lem-speciallag}. 

\begin{theorem}\label{thm-typeIIangle}
	Let $L_t$ be an almost-calibrated LMCF in $\mathbb{C}^n$ with Lagrangian angle $\theta_t$ that forms a singularity at the space-time point $(O,0)$. Assume that any sequence of Type I rescalings $L^{\sigma_i}_s$ converge as flows to the same special Lagrangian cone $C$, with angle $\overline{\theta}$. 
	
	Let $X_i=(x_i,t_i)$ be a sequence of space-time points such that $(x_i,t_i)\rightarrow (O,0)$, let $\lambda_i \in \mathbb{R}$ satisfy $-\lambda_i^2 t_i \rightarrow \infty$, and define the rescalings
	\[ L^{X_i,\lambda_i}_\tau \, := \, \lambda_i\left( L_{t_i+\lambda_i^2 \tau} - x_i \right) \]
	with Lagrangian angle $\theta_\tau^i$.
	
	Then for any bounded parabolic region $\Omega \times I \subset \mathbb{C}^n \times \mathbb{R}$,
	\begin{equation} 
		\theta^i_\tau(\chi_i) \rightarrow \overline\theta \quad \mbox{uniformly in } \Omega\times I, \label{eq-typeIIangle}
	\end{equation}
	where $\tau \in I$ and $\chi_i \in L^{X_i,\lambda_i}_\tau \cap \Omega $ is any sequence of points.
\end{theorem}

\begin{proof}
	We will be considering the following three flows:
	\begin{itemize}
		\item $L_t$, the original LMCF
		\item $L_s^{\sigma_i}$, the Type I rescaled LMCF with factor $\sigma_i$ ($\sigma_i$ to be explicitly decided later)
		\item $L_\tau^{X_i,\lambda_i}$, the LMCF rescaled around $X_i = (x_i,t_i)$ with factor $\lambda_i$.
	\end{itemize}
	The time variables $t,s,\tau$ are related by
	\[ s = \sigma_i^2 t, \, \, \, \tau = \lambda_i^2(t - t_i) = -t_i\lambda_i^2 \left(1 - \frac{s}{\sigma_i^2 t_i}\right). \]
	
	For a suitable choice of $\sigma_i$, the result (\ref{eq-typeIIangle}) can be shown using the following sequence of steps. For all $\veps$, there exists $N$ independent of $\tau, (\chi_i)_{i=1}^\infty$ such that for all $i \geq N$:
	\begin{align}
		|\theta^i_\tau(\chi_i)-\overline\theta|^2 \, &= \, 
		\int_{L_{\tau}^{X_i,\lambda_i}} |\theta - \overline{\theta}|^2 \, \Phi_{(\chi_i,\tau)} d\mathcal{H}^2 \\
		&\leq \, \int_{L_{(-t_i\lambda_i^2)(1 + t_i^{-1}\sigma_i^{-2})}^{X_i,\lambda_i}} |\theta - \overline{\theta}|^2 \, \Phi_{(\chi_i,\tau)} d\mathcal{H}^2 \label{eqn-1} \\
		&= \, \int_{L_{-\sigma_i^{-2}}} |\theta - \overline\theta|^2\Phi_{\left( \lambda_i^{-1}\chi_i + x_i,\, \lambda_i^{-2}\tau + t_i \right)} d\mathcal{H}^2 \notag\\
		&= \, \int_{L_{-1}^{\sigma_i}} |\theta - \overline{\theta}|^2 \, \Phi_{\left( \sigma_i\left( \lambda_i^{-1}\chi_i + x_i \right),\sigma_i^2\left( \lambda_i^{-2}\tau + t_i\right)\right) } d\mathcal{H}^2 \label{eqn-2} \\
		&\leq \, \int_{L_{-1}^{\sigma_i}} |\theta - \overline{\theta}|^2 \, \Phi_{(0,0)} d\mathcal{H}^2  \label{eqn-3} \, + \, \tfrac{\veps}{2}\\
		&\leq \, \veps. \label{eqn-4}
	\end{align}
	
	The idea is that we have convergence of the Type I rescalings $L_{-1}^{\sigma_i}$, as well as convergence of their Lagrangian angles. To change to an integral over $L_{-1}^{\sigma_i}$, we first change to an integral over $L_{\overline\tau}^{X_i,\lambda_i}$, for a suitable choice of $\overline\tau$, using Huisken monotonicity. \\
	
	We now proceed to justify each of these steps. To prove (\ref{eqn-1}), notice that $|\theta - \overline{\theta}|^2$ is a subsolution to the heat equation. Also, since by assumption $t_i \lambda_i^2 \rightarrow -\infty$ as $i \rightarrow \infty$, if we pick our $\sigma_i$ such that $\sigma_i \, < \, \frac{1}{2\sqrt{-t_i}}$ then \[  (-t_i \lambda_i^2)\left( 1 + \tfrac{1}{t_i \sigma_i^2} \right) \rightarrow -\infty \] for sufficiently large $i$. In particular, this quantity is eventually less than $\inf(I)$, so we may pick a uniform $N$ such that for any $i \geq N$, \[  (-t_i \lambda_i^2)\left( 1 + \tfrac{1}{t_i \sigma_i^2} \right) \leq \tau \] for all $\tau \in I$. Then we can directly apply Huisken's monotonicity formula, Theorem \ref{thm-monotonicity}. \\
	
	To prove (\ref{eqn-2}), we relate the integral over the Type II rescaling to the integral over the Type I rescaling. To do this, we apply Lemma \ref{lem-huiskenscaleinvariance} twice to relate the integral to the original flow at time $-\sigma_i^{-2}$, and then to the Type I rescaled flow at time $-1$.\\
	
	To prove (\ref{eqn-3}), we show that we can replace our spacetime-shifted heat kernel with the stationary one at $(0,0)$. As long as \[\left(\sigma_i\left( \lambda_i^{-1} \chi_i + x_i \right),\sigma_i^2\left( \lambda_i^{-2}\tau + t_i\right) \right) \rightarrow (0,0),\] we get smooth convergence of $\Phi_{\left(\sigma_i\left( \lambda_i^{-1} \chi_i + x_i \right),\sigma_i^2\left( \lambda_i^{-2}\tau + t_i\right) \right)}$ to $\Phi_{(0,0)}$. Then by Theorem \ref{thm-b} and the Type I blowup assumption, it follows that:
	\begin{align*}
		\left| \int_{L_{-1}^{\sigma_i}} |\theta - \overline{\theta}|^2 \left( \Phi_{\left(\sigma_i\left( \lambda_i^{-1}\chi_i + x_i \right),\sigma_i^2\left( \lambda_i^{-2}\tau + t_i\right) \right)} - \Phi_{(0,0)}\right) d\mathcal{H}^2  \right| \,\\
		\leq \, \left| \Phi_{\left(\sigma_i\left( \lambda_i^{-1}\chi_i + x_i \right),\sigma_i^2\left( \lambda_i^{-2}\tau + t_i\right) \right)} - \Phi_{(0,0)}\right|_{\infty} \cdot \int_{L_{-1}^{\sigma_i}} |\theta - \overline{\theta}|^2  d\mathcal{H}^2 \,
		\longrightarrow 0 .
	\end{align*}
	Convergence of the space-time points $\left(\sigma_i\left( \lambda_i^{-1}\chi_i + x_i \right),\sigma_i^2\left( \lambda_i^{-2}\tau + t_i\right) \right)$ will follow if we pick our $\sigma_i$ correctly. For example, 
	\[ \sigma_i \, := \, \frac{1}{2}\min \left \{\frac{1}{\sqrt[4]{-t_i}}, \frac{1}{\sqrt{x_i}}, \sqrt{\lambda_i} \right \} \]
	will work for this step and for step 1. All stated convergences are uniform in $\chi_i$ and $\tau$, since $\Omega$ and $I$ are bounded regions. Finally, (\ref{eqn-4}) follows from Theorem \ref{thm-b}, just as in step (\ref{eqn-3}).
\end{proof}

In particular, $-t_iA_i^2 \rightarrow \infty$ for a Type II singularity. Since the Type II rescalings converge smoothly to the limiting flow, and the Lagrangian angles of the Type I rescalings converge to $\overline\theta$ by Theorem \ref{thm-typei}, Lemma \ref{thm-typeIIangle} implies that any Type II blowup of our flow must be a special Lagrangian.

\begin{corollary}\label{cor-typeii}
	Let $L_t$ be an almost-calibrated, connected, equivariant Lagrangian MCF with Lagrangian angle $\theta_t$ that forms a singularity at the space-time point $(O,0)$. Assume that any sequence of Type I rescalings $L^{\sigma_i}_s$ converge subsequentially as flows to the same special Lagrangian cone $C$, with angle $\overline{\theta}$.
	
	Then any sequence of Type II rescalings, $L_\tau^{(x_i,t_i)}$ converges subsequentially in $C^{\infty}_{loc}$ to a special Lagrangian, with Lagrangian angle $\overline\theta$.
\end{corollary}

We are now ready to prove that any Type II blowup of our equivariant flow must be the unique Lawlor neck with asymptotic planes $P_1 \cup P_2$.

\begin{theorem}\label{thm-typeii}
	Let $L_t$ be an almost-calibrated, connected, $O(n)$-equivariant mean curvature flow in $\mathbb{C}^n$ with planar asymptotics. Then up to a translation, a Type II blowup of any finite-time singularity is a Lawlor neck $\Sigma_{\operatorname{Law}}$ with the same Lagrangian angle as the (unique) Type I blowup $P_1 \cup P_2$, and is asymptotically planar with asymptotes $P_1$ and $P_2$.
	Additionally, the Type II blowup does not depend on the rescaling sequence.
\end{theorem}

\begin{proof}
	Consider a sequence of Type II rescalings, $L_{\tau}^{(x_i,t_i)}$, that converge to a Type II blowup $L_\tau^\infty$, and denote by $A_i$ the rescaling factors. We first show that we may assume $x_i \in \mathbb{C}{\times}\{0\}^{n-1}$, so that we may apply the theory from Section \ref{sec-s1}.
	Apply a sequence of rotations $R^i(\cdot)$ centred on the origin so that $R_i(x_i) \in \mathbb{C}{\times}\{0\}^{n-1}$, and pass to a subsequence so that this sequence of rotations converges in $C^\infty$ to a rotation $R_\infty$. Then, since we are working with equivariant Lagrangians,
	\begin{align*}
		L^{(R_i(x_i),t_i)}_\tau \, &= \, A_i \left( L_{t_i + A_i^{-2}\tau} - R_i(x_i)\right)\\
		&= \, A_i \left( R_i(L_{t_i + A_i^{-2}\tau}) - R_i(x_i)\right)\\
		&= \, R_i(L^{(x_i,t_i)}_\tau) \rightarrow\, R_\infty(L^\infty_\tau),
	\end{align*}
	so up to a rotation we obtain the same limit if we use the sequence $R_i(x_i)$ instead of $x_i$.
	
	Now, we know from Corollary \ref{cor-typeii} that the Type II blowup $L^\infty_\tau$ is a special Lagrangian, i.e.\ a static flow with $\theta = \overline{\theta}$. So we only need look at $L^\infty_0$ to understand the entire flow. There are now two cases to consider: either the image of the origin $-A_ix_i$ remains bounded under the Type II rescalings, or $|A_ix_i|$ diverges to $\infty$.
	
	If $|A_ix_i| \rightarrow \infty$, then (on passing to a subsequence) $L_0^{\infty}$ is invariant under translations $T_v$ (as defined in Section \ref{sec-s1}) by Lemma \ref{lem-s1convergence}, where $v = ze_1 := \lim_{i \rightarrow \infty} \frac{-A_ix_i}{|A_ix_i|} \in \mathbb{C}{\times}\{0\}^{n-1}$, for a constant $z \in \mathbb{C}$. Therefore, referring to (\ref{eq-defTtranslation}) for the definition of $T_\nu)$ at the point $l^{\infty}_0(s)e_1 \in L^\infty_0 \cap (\mathbb{C}{\times}\{0\}^{n-1})$
	\[ \frac{\p}{\p \lambda}\bigg|_{\lambda=0} T_v (\alpha, \lambda, \gamma(s)) \, = \, -z\alpha \]
	is a tangent direction, for any $\alpha \in S^{n-2}$. In particular, $z e_i = (0,\ldots, z, \ldots, 0)$ is a tangent direction at every point in $L^\infty_0 \cap (\mathbb{C}{\times}\{0\}^{n-1})$ for all $i \neq 1$. So, if the profile curve of the Type II blowup is $l^\infty_0(s) = a(s) + ib(s)$, 
	\begin{align*}
		\arg \left( 
		\begin{vmatrix}
			a' + ib' && 0 && \cdots && 0 \\
			0 && z && \cdots && 0 \\
			\vdots && \vdots && \ddots && \vdots \\
			0 && 0 && \cdots && z
		\end{vmatrix} \right) = \overline{\theta}
		\implies \arg(a' + ib') \, = \, \overline{\theta} - (n-1)\arg(z),
	\end{align*}
	which implies that $l^\infty_0$ is a straight line through the origin, and that $L^\infty_0$ is an n-plane. But since Type II blowups must satisfy $\max|A|^2 = 1$, this is a contradiction.
	
	It follows that $|A_ix_i|$ remains bounded. In this case it follows from Lemma \ref{lem-s1convergence} and $C^{\infty}_{loc}$-convergence of the Type II rescalings that $L^{\infty}_{0}$ is an $O(n)$-equivariant submanifold of $\mathbb{C}^n$, after a translation by an element of $\mathbb{C}{\times}\{0\}^{n-1}$. Therefore by Lemma \ref{lem-speciallag}, it must be a Lawlor neck $\Sigma_{\operatorname{Law}}$.
	The uniqueness follows from Lemma \ref{lem-cone}, as there is only one Lawlor neck with $\sup|A| = 1$ and Lagrangian angle $\overline{\theta}$ that fits in the cone containing $\gamma$.
\end{proof}

\subsection{Intermediate Blowups} \label{sec-inter}

\begin{figure}[t]
	\centering
	\includegraphics[width=0.45\linewidth]{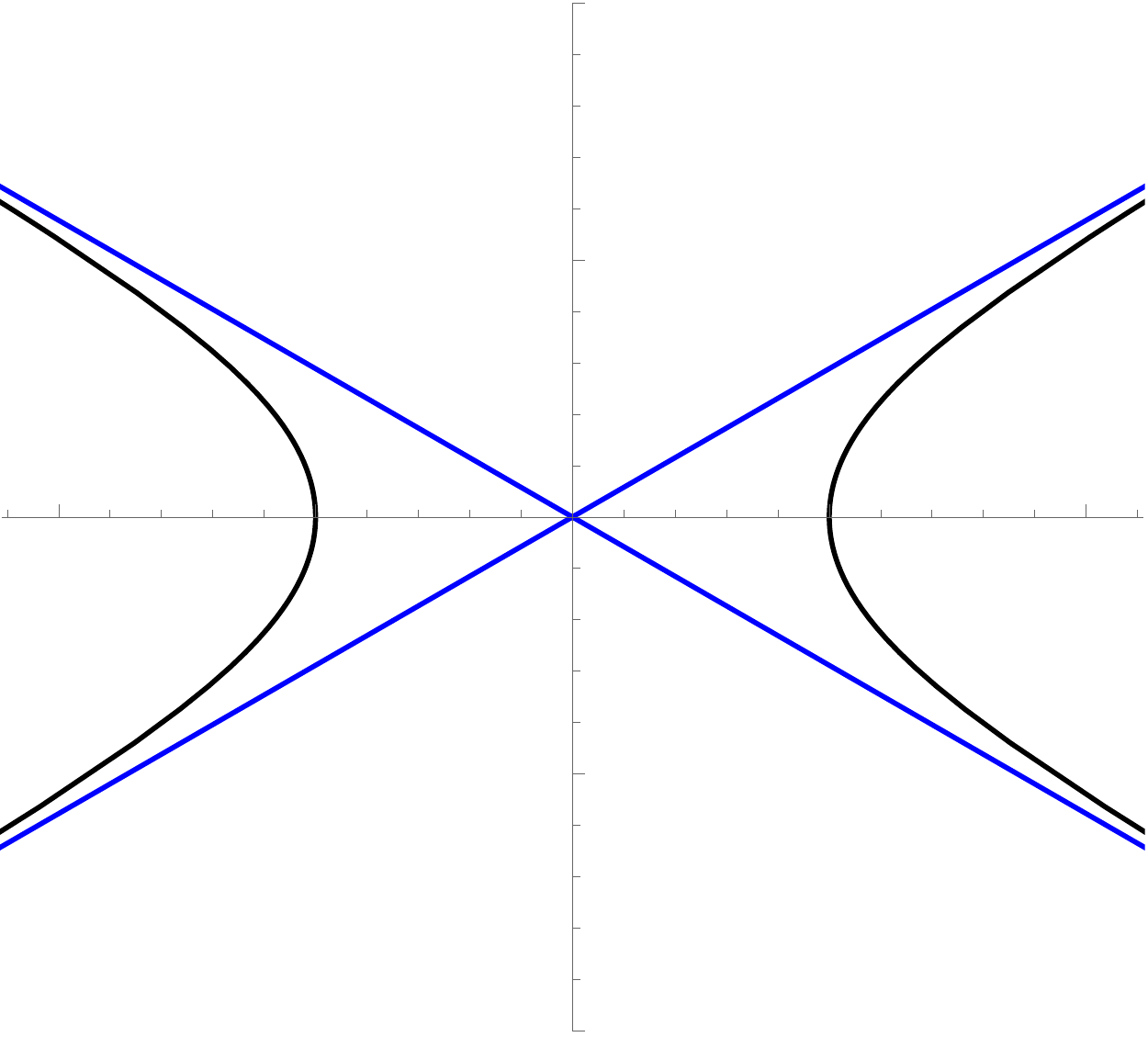}
	\caption{The profile curves of the Type I and Type II blowups for an equivariant LMCF in $\mathbb{C}^3$.}
	\label{fig-type1type2}
\end{figure}

Finally, we examine the behaviour between the Type I and Type II scales of a finite time singularity of our LMCF. Assume our flow forms a singularity at the space-time point $(0,0)$, consider a sequence of times $t_i \rightarrow 0$ from a Type II rescaling sequence (i.e.\ satisfying (\ref{eq-typeiicondition})), and let $A_i$ be the maximum value of the second fundamental form over $L_{t_i}$, as before. Let $\lambda_i \in \mathbb{R}$ be a sequence diverging to $+\infty$ such that
\[ \delta_i := \frac{\lambda_i}{A_i} \rightarrow 0, \quad -\lambda_i^2 t_i \rightarrow \infty.\]
Then we define the \textbf{intermediate rescalings} corresponding to the sequence $(t_i,\lambda_i)$ as
\[ L_\tau^{t_i,\lambda_i}\, := \, \lambda_i L_{t_i + \lambda_i^{-2}\tau}. \]
Note that we need not translate the rescaling to centre on the point of highest curvature - we proved in Theorem \ref{thm-typeii} that the origin remains bounded along the sequence, so any convergence will be unaffected by such translations. The assumptions that $\delta_i := \frac{\lambda_i}{A_i} \rightarrow 0$ and $-\lambda_i^2t_i \rightarrow \infty$ are made since otherwise the resulting blowup will just be a scaling of a Type II blowup or a time-translation of a Type I blowup respectively. We prove the following:

\begin{theorem}\label{thm-intermediate}
	Let $L_t$ be an almost-calibrated, connected, $O(n)$-equivariant mean curvature flow in $\mathbb{C}^n$ with planar asymptotics. Assume that $L_t$ forms a singularity at the origin at time $t=T$, with Type I blowup $P_1 \cup P_2$. Then for any $R, \veps$ and finite time interval $I$, there exists a subsequence such that $L_\tau^{t_i,\lambda_i} \cap (B_R{\setminus} B_\veps)$ may be expressed as a graph over $P_1 \cup P_2$ for $\tau \in I$, and this graph converges in $C^{1;0}$ to $P_1 \cup P_2$.
\end{theorem}
\begin{proof}
	Extend $I$ so that it contains 0 in its interior, and pass to a subsequence so that the Type II rescalings centred at the space-time points $(O,t_i)$ converge smoothly to a Type II blowup. By Theorem \ref{thm-typeIIangle}, on the cylinder $B_R{\times}I$ the Lagrangian angle $\theta^i$ of the intermediate rescalings is converging uniformly to a constant $\overline{\theta}$, the same value as the Lagrangian angles of the Type I and Type II blowups. For convenience we assume $\overline{\theta} \, = \frac{\pi}{2}$, that the profile curve of the Type I blowup is the pair of lines at $\alpha=\frac{\pi}{2n}$ and $\alpha=-\frac{\pi}{2n}$ (and their reflections in $O$), and that the Type II blowup is the unique Lawlor neck with $\sup|A| = 1$ asymptotic to these planes, as in Figure \ref{fig-type1type2} (this can all be achieved by a single rotation of the plane $\mathbb{C}{\times}\{0\}^{n-1}$). 
	
	If $\veps$ is small enough, and we take $i$ large enough so that $|\theta^i - \tfrac{\pi}{2}| < \veps$, then on $B_R \times I$ there is at most one intersection of each component of $l^i$ with the real axis. Denote by $b_\tau^i$ the sequence of intersections on the positive real axis at time $\tau$, where it exists, and by $\gamma_\tau^i$ the component of $l^i_\tau$ containing $b^i_\tau$. We first prove that we have the expected convergence on individual time slices.
	\begin{lemma} \label{lem-interconv}
		Fix a sequence $\tau_i \in I$. 
		\begin{itemize}
			\item If $b^i_{\tau_i} \rightarrow 0$, then for all $\veps$, the profile curves $\gamma^i_{\tau_i}$ parametrised by arc-length converge in $C^1$ on $B_R{\setminus}B_\veps$ to the half-lines at $\alpha=\frac{\pi}{2n}$ and $\alpha=-\frac{\pi}{2n}$.
			\item If $b^i_{\tau_i} \rightarrow B > 0$, then the profile curves $\gamma^i_{\tau_i}$ parametrised by arc-length or by argument converge in $C^1$ on $B_R$ to the profile curve of the Lawlor neck $\sigma_{\operatorname{law}}$, with asymptotes given by these same lines.
		\end{itemize}
	\end{lemma}
	\begin{proof}
		Throughout, we suppress the subscript $\tau_i$, as nothing depends on it.
		
		We tackle case 2 first, so $b^j \rightarrow B > 0$. Take $N$ large enough such that on $B_R {\times}I$ for $j>N$, \[|\theta^j - \tfrac{\pi}{2}| < \veps, \quad |b^j -B| < \veps. \]
		Note that, close to $\alpha=0$, by the above condition on $\theta^j$, we may parametrise $\gamma^j$ by angle:
		\begin{align*}
			\gamma^j(\alpha) \, &= \, r^j(\alpha)e^{i\alpha} \\
			\implies \dot{\gamma}^j \, &= \, (\dot{r}^j + ir^j)e^{i\alpha}\\
			\implies \dot{r}^j \, &= \, r^j\cot(\theta^j -n\alpha),	
		\end{align*}
		where $\cot:(0,\pi)\rightarrow \mathbb{R}$. In fact, by this gradient equation, we see that it is parametrisable in this fashion for 
		\[\alpha \in \left( -\frac{\pi}{2n} + \frac{\veps}{n},\frac{\pi}{2n} - \frac{\veps}{n}  \right).\] Integrating the inequality obtained by using the bound on $\theta^j$, it follows that for $\alpha > 0$:
		\[ \frac{(B-\veps)\sqrt[n]{\sin\left(\tfrac{\pi}{2} + \veps\right)}}{\sqrt[n]{\sin\left(\tfrac{\pi}{2} + \veps - n\alpha\right)}} \, \leq \, r^j(\alpha) \, \leq \, \frac{B+\veps}{\sqrt[n]{\sin\left(\tfrac{\pi}{2} - \veps - n\alpha\right)}}. \]
		This implies that, as $\varepsilon \rightarrow 0$ and $j \rightarrow \infty$,
		\[ r^j\left(\frac{\pi}{2n}-\frac{\veps}{n}\right) \rightarrow \infty. \]
		An identical argument shows that the same is true for the $\alpha < 0$ half of the curve. Therefore, on $B_R$ the curve may be fully parametrised by angle for sufficiently large $j$, and so this parametrisation converges in $C^1$ to 
		\[ r^\infty \, = \, \frac{B}{\sqrt[n]{\sin\left(\tfrac{\pi}{2} - n\alpha\right)}}, \]
		which is the Lawlor neck described in the statement of the lemma.
		
		Now assume that $b^j \rightarrow 0$. By the same method as above, we see that the curve is parametrisable by angle for the same range of $\alpha$, and for $\alpha > 0$ in this range,
		\[ \frac{b^j\sqrt[n]{\sin\left(\tfrac{\pi}{2} + \veps\right)}}{\sqrt[n]{\sin\left(\tfrac{\pi}{2} + \veps - n\alpha\right)}} \, \leq \, r^j(\alpha) \, \leq \, \frac{b^j}{\sqrt[n]{\sin\left(\tfrac{\pi}{2} - \veps - n\alpha\right)}}. \]
		Since $b^j\rightarrow 0$, for each $\veps$ we may choose $N$ large so that for $j>N$, $r^j<\veps$ on the angle range $\left( -\frac{\pi}{2n} + \veps, \frac{\pi}{2n} - \veps \right)$ and $|\theta^j - \frac{\pi}{2}| < \veps$ on $B_R {\times}I$. Therefore for $j > N$, the curve enters the cone $\Gamma := \{\alpha \in \left( \frac{\pi}{2n}- \veps, \frac{\pi}{2n} + \veps \right)\}$ within the ball $B_\veps$. Now we show that it remains there while in $B_R$ (an identical argument holds for the cone on the other side of the real axis, $\Gamma' := \{\alpha \in \left( -\frac{\pi}{2n}- \veps, -\frac{\pi}{2n} + \veps \right)\}$).
		Once the curve has entered the cone $\Gamma$, if it intersected the line $\alpha = \frac{\pi}{2n}- \veps$ again, then at this point we would have
		\begin{align*}
			\arg(\dot{\gamma}^j) \, &\leq \, \frac{\pi}{2n} - \veps \, 
			\quad \implies\quad  \theta^j \leq  \frac{\pi}{2} - n\veps,
		\end{align*}
		which is a contradiction. A similar contradiction is reached if we assume that the curve intersects the line $\alpha = \frac{\pi}{2n} + \veps$, therefore the curve must remain in the cone $\Gamma$ once it enters. 
		Now, parametrising the curve by arc-length so that
		\begin{align*}
			\gamma^j(s) \, = \, r^j(s)e^{i\alpha^j(s)} \,
			\quad \implies\, \quad \dot{\gamma}^j\, = \, e^{i (\theta^j - (n-1)\alpha^j)};
		\end{align*} 
		limiting $\veps \rightarrow 0$ shows that our curves $\gamma^j$ converge in $C^1$ away from the origin to the specified half-lines.
	\end{proof}
	
	To finish the proof, we need to show that $b^i_\tau \rightarrow 0$ uniformly in $I$. The above lemma will then show that our intermediate rescalings converge uniformly to the pair of planes we expect. We know from the Type II convergence that $b^i_0 \rightarrow 0$, and so it suffices to show that the value $b^i$ is a $C^0$-Cauchy sequence as a function of time. Intuitively, the argument is that if the Lagrangian angle is converging uniformly to a constant, then the `average' value of $H$ also is. This puts a limit on how far the profile curve can travel between times, which prevents $b^i$ converging to two different values.
	
	\begin{lemma}
		$b^i_\tau$ is a $C^0$-Cauchy sequence  of functions in $\tau$, converging to $0$.
	\end{lemma}
	\begin{proof}
		Assume for a contradiction that it isn't a Cauchy sequence. We know that $b^i_0 \rightarrow 0$, so this means that there exists $B \in \mathbb{R}^+$, $B < \frac{R}{2}$ such that, on passing to a subsequence,
		\[ \sup_{\tau \in I} |b^{i}_\tau|  > B.  \]
		Take a sequence $\tau_i \in I$ such that $b^i_{\tau_i} = B$; we assume for notational convenience that $\tau_i$ is negative. Denote by $\sigma_B$ the profile curve of the Lawlor neck intersecting the real axis at $B$, and by $v$ the two half-lines, both as described in Lemma \ref{lem-interconv}. Then by this lemma we may take $N$ sufficiently large such that for all $i \geq N$:
		\begin{itemize}
			\item $|\theta- \frac{\pi}{2}|<\veps$ in $B_R{\times}I$,
			\item $d_{Haus}(\gamma^{i}_{\tau_i}\cap B_R, \, \sigma_B \cap B_R) < \veps$,
			\item $d_{Haus}(\gamma^{i}_{0} \cap B_R, \, v \cap B_R) < \veps$.
		\end{itemize}
		If we denote by $B_\veps(A)$ the $\veps$-fattening of the set $A$, then this means that for $i \geq N$, $\gamma^{i}_0 \subset B_\veps(v)$ and $\gamma^{i}_{\tau_i} \subset B_\veps(\sigma_B)$. Let $d(\veps) \, := \, d_{Haus}(B_\veps(\sigma_B)\cap B_R,\, B_\veps(v)\cap B_R)$, and notice that it is a decreasing function of $\veps$.
		
		Taking $\varepsilon$ sufficiently small, we may find $p_1\leq p_2 \in \mathbb{R}$ such that $\gamma^i_\tau(p_1),\gamma^i_\tau(p_2) \in B_R$ for all $\tau \in I$. Such points must exist, else by an identical argument to the one given below, the integral of $\vec{H}$ over an escaping region of the curve would be large, contradicting the uniform bound on $\theta$. Now take $\nu$ to be the outward pointing normal and $s$ the arc-length parameter. Then since the flow must travel from $B_\veps(\sigma_B)$ to $B_\veps(v)$ between times $0$ and $\tau_i$, it follows by the definition of mean curvature flow that
		\begin{align*}
			\forall p \in [p_1,p_2], \, \, \, \int_{\tau_i}^0 H(p)\cdot \nu(p) \, d\tau \, &\geq \, d(\veps) \\
			\implies \int_{\tau_i}^0 \int_{\gamma_\tau^i([p_1,p_2])} H \cdot \nu \, ds \, d\tau \, &\geq \, d(\veps) \cdot \min_{\tau \in [\tau_i,0]} \mathcal{H}^1(\gamma^i_{\tau}[p_1,p_2]).
		\end{align*}
		But on the other hand, since $\vec{H} = J\nabla\theta$,
		\begin{align*}
			\left| \int_{\tau_i}^0 \int_{\gamma_\tau^i([p_1,p_2])} H \cdot \nu \, ds \, d\tau \, \right| \, &= \, \left| \int_{\tau_i}^0 \int_{\gamma_\tau^i([p_1,p_2])} \frac{\p \theta^i_\tau}{\p s} \, ds \, d\tau \, \right| \\ \,&=\, \left| \int_{\tau_i}^0 \theta^i_\tau(p_2) - \theta^i_\tau(p_1) \, d\tau \, \right| \,	\leq \, -2\tau_i \veps,
		\end{align*}
		which is a contradictory inequality if $\varepsilon$ is taken small.
	\end{proof}
	This completes the proof of Theorem \ref{thm-intermediate}.
\end{proof}

\bibliography{paper7, library} 
\bibliographystyle{alpha}

\end{document}